\documentclass[a4paper,11pt]{article}
\usepackage[utf8]{inputenc}
\usepackage{amsmath}
\usepackage{amsfonts}
\usepackage{amssymb}
\usepackage{amsthm}
\usepackage[english]{babel}
\usepackage{fontenc}
\usepackage{graphicx}
\usepackage{subfigure}
\usepackage{dsfont}
\usepackage[hmargin=3cm,vmargin=2.5cm,paper=a4paper]{geometry}
\usepackage{hyperref}
\usepackage{color}

\newcommand{\ls}{\leqslant}
\newcommand{\gs}{\geqslant}
\renewcommand{\div}{\operatorname{div}}
\newcommand{\R}{\mathbb R}
\newcommand{\per}{\operatorname{Per}}
\newcommand{\eps}{\varepsilon}
\newcommand{\BV}{\mathrm{BV}}

\newcommand{\dd}{\, \mathrm{d}}
\newcommand{\Id}{\mathrm{Id}}
\newcommand\restr[2]{{\left.\kern-\nulldelimiterspace #1 \vphantom{\big|} \right|_{#2}}}

\newcommand{\TV}[1]{\mathrm{TV}{(#1)}}
\newcommand{\TVO}[1]{\mathrm{TV}{(#1\, ; \, \Omega)}}
\renewcommand{\L}{\mathrm{L}}
\newcommand{\W}{\mathrm{W}}

\newcommand{\Ltwou}{{\mathrm{L}^2(\Omega)}}
\newcommand{\Ltwof}{{\mathrm{L}^2(\Sigma)}}
\newcommand{\Fd}{\mathcal{F}_\alpha}
\newcommand{\Fn}{\hat{\mathcal{F}}_\alpha}
\newcommand{\Da}[1]{\mathcal{D}_\alpha(#1)}
\newcommand{\Br}{B(x,r)}
\newcommand{\ua}{u_\alpha}
\newcommand{\va}{v_\alpha}
\newcommand{\kaw}{\kappa_{\alpha, w}}
\newcommand{\pa}{p_\alpha}
\newcommand{\range}[1]{\mathcal{R}(#1)}
\newcommand{\abs}[1]{\lvert#1\rvert}
\newcommand{\norm}[1]{\left\|#1\right\|}
\newcommand{\set}[1]{\left\{#1\right\}}
\newcommand{\inner}[2]{\left<#1,#2\right>}
\newcommand{\supp}{\mathop{\mathrm{supp}}}
\newcommand{\udag}{u^\dagger}

\DeclareMathOperator{\Rcirc}{R_{circ}}
\DeclareMathOperator{\Rcircadj}{R_{circ}^*}

\newcommand{\scal}[2]{ \left \langle #1, \ #2 \right \rangle}
\newtheorem{theorem}{Theorem}
\newtheorem{prop}{Proposition}
\newtheorem{lemma}{Lemma}

\newtheorem{definition}{Definition}
\theoremstyle{remark}
\newtheorem{remark}{Remark}

\begin{document}
\title{A note on convergence of solutions of total variation regularized linear inverse problems}
\author{Jos\'e A. Iglesias\thanks{Johann Radon Institute for Computational and Applied Mathematics (RICAM), Austrian Academy of Sciences, Linz, Austria (\texttt{jose.iglesias{@}ricam.oeaw.ac.at}, \texttt{gwenael.mercier{@}ricam.oeaw.ac.at})} , Gwenael Mercier\footnotemark[1] , Otmar Scherzer\footnotemark[1] \thanks{Computational Science Center, University of Vienna, Austria (\texttt{otmar.scherzer{@}univie.ac.at})}}
\date{}
\maketitle
\begin{abstract}
In a recent paper by A. Chambolle et al. \cite{ChaDuvPeyPoo17} it was proven that if the subgradient of the total variation at the noise free data is not empty, the level-sets of the total variation denoised solutions converge to the level-sets of the noise free data with respect to the Hausdorff distance. The condition on the subgradient corresponds to the source condition introduced by Burger and Osher \cite{BurOsh04}, who proved convergence rates results with respect to the Bregman distance under this condition. We generalize the result of Chambolle et al. to total variation regularization of general linear inverse problems under such a source condition. As particular applications we present denoising in bounded and unbounded, convex and non convex domains, deblurring and inversion of the circular Radon transform. In all these examples the convergence result applies. Moreover, we illustrate the convergence behavior through numerical examples.
\end{abstract}

\section{Introduction}
In this paper we are concerned with total variation regularization of linear inverse problems 
\begin{equation} 
\label{eq:Auf}
Au=f\,,
\end{equation}
for functions $u: \Omega \to \R$, where 
\begin{itemize}
\item $\Omega = \R^2$ or 
\item $\Omega$ is a bounded Lipschitz domain $D \subset \R^2$,
\end{itemize}
and $A : \Ltwou \to \Ltwof$ is a linear bounded (typically compact) operator. 

Since in general the solution of \eqref{eq:Auf} is ill-posed, some sort of regularization needs to be employed.
The method considered in this paper is total variation regularization, in which a regularization parameter $\alpha > 0$ is chosen, and either of the two following minimization problems is solved:
\begin{itemize}
\item The Dirichlet (resp. full space) problem consisting in computation of the minimizer of the functional
\begin{equation} 
\Fd(u) := \frac{1}{2}\norm{Au-f}_\Ltwof^2 + \alpha \TV{u}  \label{eq:primalDirichlet}
\end{equation}
over either one of the sets of functions
\begin{equation*}
u \in \L^2(D) \text{ or } u \in \L^2(\R^2)\;.
\end{equation*}
Here $\TV{u} \in [0, +\infty]$ denotes the total variation in $\R^2$ of the function $u$ (or of its extension by zero). 

\item The Neumann problem consists in minimizing 
\begin{equation}\label{eq:primalNeumann}
\Fn(u):=\frac{1}{2}\norm{Au-f}_\Ltwof^2 + \alpha\TVO{u} \text{ over }u \in \L^2(\Omega)\;.
\end{equation}
Here $\TVO{u}$ is the total variation of $u$ computed in the open set $\Omega$.
\end{itemize}

Previously, Chambolle et al \cite{ChaDuvPeyPoo17} proved that in the full space problem for $A=\Id$ the level-sets of regularized solutions converge in Hausdorff distance to those of $f$. The main goal of this paper is to show that the techniques of \cite{ChaDuvPeyPoo17} apply to TV-regularization for general linear ill-posed problems. We therefore show that the level-sets of the TV-regularized solutions Hausdorff converge assuming a source condition and adequate parameter choices. Furthermore, we consider different boundary conditions (Dirichlet, Neumann and full space) which not only have significant impact (more than in the denoising case) on the reconstructions, but also require new ideas for the proofs.

One of the main reasons for the use of total variation regularization is dealing with solutions which contain discontinuities, in particular piecewise constant functions representing the contours of separate objects. Hausdorff convergence of level-sets is in fact particularly well suited to such a situation, since it means (see Definition \ref{def:hausdist}) that the maximal distance between points of the regularized objects and those of the noise free solutions goes to zero. Such a convergence is desirable both in imaging applications, where it has a direct visual interpretation, as well as in identification of inclusions, where it corresponds to uniform convergence of the inclusions themselves.

\paragraph{Structure of the paper.}
The outline of the paper is as follows: In Section \ref{sec:ca} we prove existence of minimizers of $\Fd$ and $\Fn$, review dual formulations of the corresponding optimization problems and explore the convergence of dual solutions (see \eqref{eq:duality}) under vanishing data perturbations, while assuming the source condition. In Section \ref{sec:cvls}, we see that the curvature of level-sets of minimizers is strongly linked to these dual variables and we explain (following \cite{ChaDuvPeyPoo17}) how the convergence of the curvatures implies the main result on Hausdorff convergence of level-sets. In Section \ref{sec:density} we give a proof, for each of the different boundary conditions considered, of the main ingredient needed for the convergence: density estimates \eqref{eq:densityest} for the level-sets. Finally, Section \ref{sec:examples} contains some inverse problems examples, where the results of previous sections apply. We discuss the effect of boundary settings on the regularized solutions. Moreover, some numerical results are presented.

\subsection{Notation and spaces}
We recall that the total variation of a function $u \in \L^1_{\text{loc}}(\Omega)$ is defined by
\begin{equation}\label{eq:defTV}\TVO{u} := \abs{Du}(\Omega) =  
 \sup \left\{\int_\Omega u~ \div z~dx \, \middle\vert\, z \in C^\infty_0(\Omega\, ;\,  \R^2),\|z\|_{\L^\infty(\Omega)} \ls 1\right\}.\end{equation}
If $\TVO{u}$ is finite, then the distributional derivative $Du$ is a vector-valued Radon measure on $\Omega$. We also emphasize that for $\Omega = \R^2$ we write
\begin{equation*}
\TV{u}:= \TV{u;\R^2},
\end{equation*}
and that for $\Omega=\R^2$, the Dirichlet and Neumann problems coincide.

We also recall that for every Lebesgue measurable $E \subset \Omega$ the perimeter of $E$ in $\Omega$ is defined to be
$$ \per(E \, ;\, \Omega) := \TVO{1_{E}},$$
where $1_E$ is the characteristic function of $E$, that is, $1_E(x)=1$ if $x \in E$, and $1_E(x)=0$ otherwise. If this quantity is finite, $E$ is said to have finite perimeter.
When $\Omega = \R^2$ or when $\Omega$ is clear from the context we skip the second argument in the above notation.

If $\Omega$ is a bounded domain $D$, we can identify $\L^2(D)$ with the set of extended functions $\set{ u \in \L^2(\R^2) \, \middle \vert \, \text{supp}(u) \subset \overline D}$ and since we have the inclusion $\L^2(D) \subset \L^1(D),$ candidates for minimizers of $\Fd$ are in
\[\set {v \in \BV(\R^2) \ \middle \vert \ v \equiv 0 \text{ on } \R^2 \setminus D},\]
where we adopt the standard definition
\[\BV(\Omega):=\set{ u \in \L^1(\Omega) \ \middle \vert \ \TV{u;\Omega} < +\infty}.\]

This corresponds to assuming a homogeneous Dirichlet boundary condition and possible jumps at the boundary are taken into account. On the other hand, if $\Omega = \R^2$, minimizers of $\Fn$ and $\Fd$ are identical, and 
are elements of the set
\[\set{ u \in \L^2(\R^2) \, \middle \vert \, \TV{u} < +\infty}.\]
When minimizing $\Fn$, corresponding to the homogeneous Neumann condition, the natural space is $\BV(\Omega) \cap \Ltwou$. The influence of $\Omega$ and the boundary conditions on the solutions is discussed in Section \ref{sec:examples}.

We stress that the minimization problems for the functionals \eqref{eq:primalDirichlet} and \eqref{eq:primalNeumann} that we deal with in all of this paper are considered over $\L^2(\Omega) \subset \L^1_{\text{loc}}(\Omega)$, over which the total variation, as defined in \eqref{eq:defTV} may be $+\infty$. In correspondence, we will consider the subgradient of a functional $F: \L^2(\Omega) \to \R \cup \{+\infty\}$, defined by
\begin{equation*}\label{eq:subgrad}\partial F(u) := \left\{ v \in \L^2(\Omega) \, \middle\vert\, F(u+h)-F(u) \gs \scal{v}{h}_{\L^2(\Omega)}\text{ for all }h \in \L^2(\Omega)\right\}.\end{equation*}
In particular, if $F$ is proper (not identically $+\infty$) and for some $u$ we have $F(u)=+\infty$, then $\partial F(u)= \emptyset$.

\section{Dual solutions and source condition}\label{sec:ca}

\begin{prop} The functional $\Fd$ defined in \eqref{eq:primalDirichlet} has at least one 
minimizer in $\L^2(\Omega)$. If $A$ is injective, there exists a unique minimizer. 
\label{prop:existmin}
\end{prop}
\begin{proof}Let $(u_k)$ be a minimizing sequence for $\Fd$. Since $u_k \in \Ltwou$ implies $u_k \in \L^1_{\text{loc}}(\R^2)$ and we work in dimension $2$, we can use Sobolev's inequality \cite[Theorem 3.47]{AmbFusPal00} to get
\[\norm{u_k}_\Ltwou \ls C \, \TV{u_k}.\]

Now, the right hand side is bounded uniformly in $k$ so that the Banach-Alaoglu theorem for $\Ltwou$ and a compactness result \cite[Theorem 3.23]{AmbFusPal00} provide a subsequence $(u_k)$ (not relabelled) that converges both weakly in $\Ltwou$ and strongly in $\L^1_{\text{loc}}(\R^2)$ to some limit $\ua$. Since $A$ is a bounded linear operator, $Au_k$ also converges weakly to $A\ua$ in $\Ltwof$. Lower semicontinuity of the norm with respect to weak convergence, and of the total variation with respect to strong $\L^1_{\text{loc}}(\R^2)$ convergence \cite[Remark 3.5]{AmbFusPal00} proves that $u_\alpha$ is a minimizer of $\mathcal{F}_\alpha$.

The uniqueness statement is straightforward, since $\norm{\cdot}_\Ltwof^2$ is strictly convex.
\end{proof}

\begin{remark}
Minimization of $\Fn$ and $\Fd$ can produce markedly different results. An example is the choice $\Omega=(-1,1)^2$, $\Sigma=(-1+\eta,1-\eta)^2$  for some $\eta \in (0,1)$, $\alpha=1$, $f=0$ and $A$ defined by 
\[Au(x)=u(x) - \frac{1}{4\eta^2}\int_{(-\eta, \eta)^2} u(x+y) \dd y,\]
continuous since $\norm{Au}_\Ltwof \ls 2 \norm{u}_\Ltwou$ by the triangle inequality and Young's inequality for convolutions. In this situation, the functional \eqref{eq:primalNeumann} is not coercive: considering the sequence $u_n := n 1_\Omega$, we have that $\hat{\mathcal{F}}(u_n)=0$ for all $n$, but $u_n$ is not bounded in $\L^2(\Omega)$. The underlying reason is that constant functions 
are annihilated by $A$, that is, $A\mathds{1} = 0$, where $\mathds{1}$ represents the constant function with value $1$ (this situation has also been discussed in \cite{Ves01}). In contrast, when working with Dirichlet boundary we have $\TV{u_n}=4n$. Note that in the denoising case ($A=\Id$) the data term makes the functional coercive in $\L^2(\Omega)$ even when using $\TVO{u}$.
\label{rem:boundary}
\end{remark}

\begin{prop}
The functional $\Fn$ defined in \eqref{eq:primalNeumann}, considered in $\Ltwou$, has at least one minimizer. 
If $A$ is injective, there exists a unique minimizer. 
\end{prop}
\begin{proof}
As noticed in Remark \ref{rem:boundary}, the situation is slightly different from Prop. \ref{prop:existmin}. Indeed, if as above, $(u_k)$ is a minimizing sequence for $\Fn$, Poincar\'e inequality gives the existence of a constant $m_k$ such that
\begin{equation} 
\Vert u_k - m_k \Vert_{\Ltwou} \ls C \TV{u_k} \ls C. \label{eq:soboTV}
\end{equation}
Now, if the constant functions are annihilated by $A$ (that is, $A\mathds{1} = 0$), then $A u_k -f = A (u_k-m_k) - f$ and 
$\TVO{u_k - m_k}=\TVO{u_k}$, so that $v_k := u_k - m_k$ is also a minimizing sequence. Since $v_k$ is bounded in $\L^2(\Omega)$ by 
\eqref{eq:soboTV}, it converges weakly to some $v \in \Ltwou$. Similarly to Prop. \ref{prop:existmin}, one can use compactness and lower semicontinuity to show that $v$ is a minimizer of \eqref{eq:primalNeumann}. \\
On the other hand, if $A \mathds{1} \neq 0$, the minimizing property for $u_k$ implies that $(Au_k -f)$ is uniformly bounded in $k$. Therefore, so is $(Au_k)$. The Poincar\'e inequality implies that $A(u_k - m_k)$ is also bounded uniformly in $k$, which forces $(Am_k)$ to be bounded too. The boundedness of $A$ gives 
$$ \Vert A \cdot m_k \Vert_{\Ltwof} = |m_k| \Vert  A\mathds{1} \Vert_{\Ltwof},$$
and since the left hand side is bounded, the sequence $|m_k|$ is also bounded and therefore, by \eqref{eq:soboTV}, $(u_k)$ is bounded in $\L^2(\Omega)$ and converges weakly (up to a subsequence) to some $u$. The end of the proof works then again as in Prop. \ref{prop:existmin}.
\end{proof}

In the rest of the section, we assume that we are in the case of $\Omega = \R^2$ or Dirichlet boundary conditions, but the results and their proofs are identical for Neumann boundary conditions.

First, we recall some basic results about the convergence of $u_\alpha$ as $\alpha$ vanishes, when some noise is added to the data $f$.
\begin{lemma} \label{lem:uconv}
Let $A: \Ltwou \to \Ltwof$ be a bounded linear operator. 
Moreover, assume that there exists a solution $\tilde{u}$ of \eqref{eq:Auf} which satisfies 
$\TV{\tilde{u}} < \infty$. 
Then the following results hold:
\begin{itemize}
 \item There exists a solution $\udag$ of \eqref{eq:Auf} with minimal total variation. That is $A \udag = f$ and 
       \begin{equation*}\label{eq:min}
        \TV{\udag} = \inf \set{ \TV{u} \, \middle\vert \, u \in \Ltwou \mbox{, } A u = f}\,. 
       \end{equation*}
  \item Given a sequence $(\alpha_n)$ with $\alpha_n \to 0^+$, elements $w_n \in \Ltwou$ and some positive constant $C$ such that 
\begin{equation}\label{eq:paramChoiceConv}\frac{\norm{w_n}_{\Ltwof}^2}{\alpha_n} \ls C,\end{equation}
there exist a (not relabelled) subsequence $(\alpha_n)$ and minimizers $(u_{\alpha_n, w_n})$ of 
$$u \mapsto \frac{1}{2}\norm{Au-(f+w_n)}_\Ltwof^2 + \alpha_n \TV{u}$$
such that $u_{\alpha_n, w_n} \rightharpoonup \udag$ weakly in $\Ltwou$ for $u^\dag$ a solution of \eqref{eq:Auf} with minimal total variation. Additionally, this convergence is also strong with respect to the $\L^1_{\text{loc}}(\R^2)$ and $\L^p(\Omega)$ 
for $1 \ls p <2$, topology, respectively, if $\Omega$ is bounded. Furthermore, $\TV{u_{\alpha_n,w_n}} \to \TV{\udag}$.
\end{itemize}
\end{lemma}
\begin{proof}
A proof of the first statement can be found in \cite[Theorem 3.25]{SchGraGroHalLen09}. The second relies on the compactness of the embedding of $\BV$, and may be found in \cite[Theorem 3.26]{SchGraGroHalLen09} (see also \cite[Theorem 5.1]{AcaVog94}).
\end{proof}

For the results contained in the rest of this section, we consider the noiseless case, and therefore we denote a generic minimizer of $\Fd$ with the fixed data $f$ by $\ua$. We now introduce the source condition, which is the key assumption for our results.

\begin{definition}\label{de:source}
Let $A^\ast: \Ltwof \to \Ltwou$ be the adjoint of $A$. We say that a minimum norm solution 
$\udag$ satisfies the \emph{source condition} if 
\begin{equation}\label{eq:srccnd}
\range{A^\ast} \cap \partial \TV{\udag} \neq \emptyset.
\end{equation}
Here $\range{A^\ast}$ denotes the range of the operator $A^\ast$ and $\partial \TV{\udag}$ denotes the subgradient 
of $\TV{\cdot}$ at $\udag$ with respect to $\Ltwou$. 
\end{definition} 

\begin{remark}
This source condition, first introduced in \cite{BurOsh04}, is standard in the inverse problems community. It is the natural condition to obtain convergence rates (with respect to the Bregman distance) of $\ua \to u^\dag$. See \cite[Proposition 3.35, Theorem 3.42]{SchGraGroHalLen09}.
\end{remark}

\begin{remark} Let us notice that the set in \eqref{eq:srccnd} does not depend on which minimal variation solution $\udag$ is chosen. Indeed, let $\udag_1, \udag_2$ be such solutions and assume 
\[A^*p \in \partial \TV{\udag_1},\] which means that for every $h \in \Ltwou$
\[\TV{\udag_1+h}-\TV{\udag_1} \gs \inner{A^*p}{h}_\Ltwou.\]
Now we can write for every $k \in \Ltwou$, since $\TV{\udag_2}=\TV{\udag_1}$
\begin{align*}\TV{\udag_2+k}-\TV{\udag_2}&=\TV{\udag_1+(\udag_2-\udag_1)+k}-\TV{\udag_1}\\
&\gs \inner{A^*p}{(\udag_2-\udag_1)+k}_\Ltwou \\
&= \inner{A^*p}{\udag_2-\udag_1}_\Ltwou + \inner{A^*p}{k}_\Ltwou\\
&= \inner{p}{A\udag_2-A\udag_1}_\Ltwof+\inner{A^*p}{k}_\Ltwou\\
&=\inner{A^*p}{k}_\Ltwou,
\end{align*}
which means that $A^*p \in \partial \TV{\udag_2}$.
\end{remark}

\begin{theorem} \label{th:EkeTem}
\begin{itemize}
 \item Let $\alpha > 0$. The dual problem (in the sense of \cite{EkeTem99}) of minimizing the functional $\Fd$, defined in \eqref{eq:primalDirichlet}, on  
 $\Ltwou$ consists in maximizing, among $p \in \Ltwof$ such that $A^\ast p \in \partial \TV{0}$, the quantity
 \begin{equation}
 \label{eq:dual}
 \Da{p} := \inner{f}{p}_\Ltwof - \frac{\alpha}{2} \norm{p}_\Ltwof^2.  
 \end{equation}
 Moreover, 
 \begin{equation}
 \label{eq:duality}
 \inf_{u \in \Ltwou} \Fd(u) = \sup_{A^\ast p \in \partial \TV{0}} \Da{p}. 
 \end{equation}
 If these quantities are attained by $\ua$, $\pa$, then we have the extremality relations
\begin{equation} 
A^\ast \pa \in \partial \TV{\ua}  \label{eq:extcond1} 
\end{equation}
and
\begin{equation*} 
\pa \in -\partial \left( \frac{1}{2\alpha} \norm{A \cdot -f}_\Ltwof^2 \right)(\ua) = \set{\frac 1\alpha (f-A\ua)}. \label{eq:extcond2} \end{equation*}
 \item Similarly, the formal limits of the minimization problems for \eqref{eq:primalDirichlet} and \eqref{eq:dual} when $\alpha \to 0$ write
 \begin{equation}
  \inf \{ \TV{u} \ \mid \ u\in \Ltwou, \, Au = f \}
  \label{eq:primal0}
 \end{equation}
 and
 \begin{equation} \label{eq:dual0}
 \sup_{A^\ast p \in \partial \TV{0}} \scal{p}{f}_\Ltwof = \sup_{v \in \range{A^\ast} \cap \partial \TV{0}} \scal{v}{u^\dag}_\Ltwou .
 \end{equation}

The extremality conditions for \eqref{eq:primal0} and \eqref{eq:dual0}, provided the quantities above are attained by some $u^\dag, p_0$, write
\begin{equation} 
A^\ast p_0 \in \partial \TV{u^\dag}  \label{eq:extcond01} 
\end{equation}
and
\begin{equation*} 
p_0 \in - \left( \partial\chi_{\{f\}} (A \,\cdot ) \right)(u^\dag) = \Ltwof, \label{eq:extcond02} \end{equation*}
\end{itemize}
where $\chi_{\{f\}}$ is the indicator function (defined on $\Ltwof$) of the set $\{f\}$, i.e. $\chi_{\{f\}}(q)=0$ if $q=f$, and $\chi_{\{f\}}(q)=+\infty$ otherwise.
\end{theorem}
\begin{proof} In the $\L^2$ setting we can make use of classical Fenchel duality, applying Theorem \ref{thm:fenchel} in the appendix with the choices
\begin{equation*} \label{eq:EkeTem}
 X=\Ltwou, Y=\Ltwof, A = A, F(\cdot) = \TV{\cdot} \mbox{ and } G(\cdot) = \frac{1}{2\alpha} \norm{\cdot -f}_\Ltwof^2.
\end{equation*}
We now check the assumptions of this theorem: $F$ and $G$ are convex and lower semi-continuous in $\Ltwou$. In addition, there exists $v \in \Ltwou$, for instance $v = 0$, such that $\TV{v} < + \infty$, $\frac{1}{2\alpha} \norm{Av -f}_\Ltwof^2 < +\infty$ and $w \mapsto \frac{1}{2\alpha} \norm{w -f}_\Ltwof^2 < +\infty$ is continuous at $Av$.

Now, noticing that since by definition $\mathrm{TV}$ is the conjugate of the indicator function of the set 
$$\mathcal K = \left\{\div z \, \middle\vert\, z \in C^\infty_0(\Omega\, ;\,  \R^2),\|z\|_{\L^\infty(\Omega)} \ls 1\right\},$$
we get that its conjugate $\mathrm{TV}^*$ is the indicator function of the closure $\overline{\mathcal K}$ of $\mathcal K$ in the $\L^2$ topology \cite[Propositions I.4.1 and I.3.3]{EkeTem99}. On the other hand, we have \cite[Proposition I.5.1]{EkeTem99} that $v \in \partial\TV{0}$ if and only if $\mathrm{TV}^*(v)=0$, that is, when $v \in \overline{\mathcal K}$.
\end{proof}

The assumption that there exists a maximizer of \eqref{eq:dual0} is in fact related to 
the source condition \eqref{eq:srccnd}:

\begin{lemma}\label{lem:sourceiffdual}
The following identity holds:
\begin{equation} 
       \partial\TV{\udag} = \set {v \in \partial \TV{0} \ \middle \vert \ \inner{v}{\udag}_\Ltwou = \TV{\udag}}.\label{eq:gtv} 
\end{equation}
Furthermore, there exists $p_0$ maximizing the functional defined in \eqref{eq:dual0} over $p \in \Ltwou$ satisfying $A^\ast p \in \partial \TV{0}$ if and only if the source condition \eqref{eq:srccnd} is satisfied. 
\end{lemma}
\begin{proof}
The identity \eqref{eq:gtv} follows directly from Lemma \ref{lem:subgrad} proved in the appendix. For the second part, we start by noticing that
\begin{equation} \label{eq:h1}
 \inner{p}{f}_\Ltwof = \scal{p}{A\udag}_\Ltwof = \inner{A^\ast p}{\udag}_\Ltwou.
\end{equation}
The source condition \eqref{eq:srccnd} implies the existence of $p_0 \in \Ltwof, v_0 \in \Ltwou$ such 
       that 
       \begin{equation*}
       v_0 = A^\ast p_0 \in \partial \TV{\udag}.
       \end{equation*}
       Then, we note that for an arbitrary $v \in \partial \TV{0}$ the definition of the subgradient implies that 
       $\TV{\udag} - \TV{0} - \inner{v}{\udag-0}_\Ltwou \geq 0$, and thus
       \begin{equation} \label{eq:he1}
       \inner{v}{\udag}_\Ltwou \ls \TV{u^\dag}.
       \end{equation}
In particular, \eqref{eq:gtv} and \eqref{eq:he1} imply that $v_0 \in \partial \TV{0}$ and maximizes $\partial \TV{0} \ni v \mapsto \inner{v}{\udag}_\Ltwou$. Therefore, using \eqref{eq:h1}, it follows that $p_0 \in\Ltwof$ is a maximizer of the functional defined in \eqref{eq:dual0}.

Conversely, if $p_0 \in \Ltwof$ maximizes $\scal{\cdot}{f}_\Ltwof$ among $p$ such that $A^\ast p \in \partial \TV{0}$, then the extremality condition \eqref{eq:extcond01} ensures that
$$ A^\ast p_0 \in \partial \TV{\udag},$$
and thus the source condition is satisfied.
\end{proof}

\begin{remark}
The minimizers of the primal functionals \eqref{eq:primalDirichlet}, \eqref{eq:primal0} as well as the maximizers of the limit dual 
functional \eqref{eq:dual0} are not unique in general. However, the dual functional $\mathcal D_\alpha$, defined in \eqref{eq:dual}, 
has a unique maximizer $p_\alpha$. Indeed, the existence follows directly since $\partial\TV{0}$ is weakly closed (subgradients of lower semicontinuous convex functions 
are convex and strongly closed \cite[Proposition 16.4]{BauCom17}, hence weakly closed) and non empty (zero is for example in it).
Uniqueness follows by the strict convexity of the squared $\mathrm{L}^2$ norm and convexity of $\partial\TV{0}$.
\end{remark}

The following proposition is a key result explaining the importance of the source condition, and its influence on the behavior of the dual solutions. The arguments are similar as proving convergence of the Augmented Lagrangian Method (see \cite{Glo84}), which have also been used to prove convergence rates results for dual variables of Tikhonov regularized solution \cite{FriSch10} and to prove existence of Bregman TV-flow
\cite{BurFriOshSch07}. The proof of the first part of this proposition follows \cite{DuvPey15}.
\begin{prop}
 Let the source condition \eqref{eq:srccnd} be satisfied and let $\pa$ be the maximizer of \eqref{eq:dual}. Then,
 $$ \lim_{\alpha \to 0^+} \pa = p_m \quad \mbox{strongly in }\Ltwof,$$
 where $p_m$ is the maximizer of $\eqref{eq:dual0}$ with minimal $\Ltwof$ norm. 
 Conversely, if $(\pa)$ is bounded in $\Ltwof$, then the source condition is satisfied (and therefore $(\pa)$ is also convergent).
 \label{prop:conv}
\end{prop}
\begin{proof}
Let $p_0$ be a maximizer of \eqref{eq:dual0}, which exists by Lemma \ref{lem:sourceiffdual}. We have that
$$ \inner{p_0}{f}_\Ltwof = \inner{A^\ast p_0}{\udag}_\Ltwou \gs \inner{A^\ast \pa}{\udag}_\Ltwou$$
and analogously, since $\pa$ maximizes $\Da{\cdot}$ that
\begin{equation}\label{eq:comp2}
\inner{A^\ast \pa}{\udag}_\Ltwou - \frac{\alpha}{2} \norm{\pa}_\Ltwof^2 \gs \inner{A^\ast p_0}{\udag}_\Ltwou - \frac{\alpha}{2} \norm{p_0}_\Ltwof^2. 
\end{equation}
Summing these inequalities, we see that $(\pa)$ is bounded and therefore converges weakly (up to a subsequence) to some $p_m \in \Ltwof$. Passing to the limit in the two previous equations gives
$$\inner{p_m}{f}_\Ltwof = \inner{p_0}{f}_\Ltwof$$
where $A^\ast p_m \in \partial \TV{0}$, the latter being weakly closed. Equation \eqref{eq:comp2} and weak convergence imply that 
$$ \norm{p_m}_\Ltwof \ls   \liminf \norm{\pa}_\Ltwof \ls \norm{p_0}_\Ltwof$$ 
which implies that $p_m$ is actually the minimal norm maximizer of the functional $\inner{\cdot}{f}_\Ltwof$ over $p$ such that $A^\ast p \in \partial \TV{0}$, and that the convergence is strong (and for every subsequence).

Let us now assume that $(\pa)$ is bounded in $\Ltwof$. Then by weak compactness for the $p_\alpha$ and applying Lemma \ref{lem:uconv} (with $w_n = 0$), there exist $\alpha_n \to 0$, $u^\dag$ a solution of $Au = f$ with minimal total variation, and $\overline p$ such that 
\begin{align*}p_{\alpha_n} &\rightharpoonup \overline p\text{ in }\Ltwof, \text{ and }\\
u_{\alpha_n} &\to u^\dag\text{ in }\L^1_{\text{loc}}(\Omega).
\end{align*}
The extremality conditions \eqref{eq:extcond1} and \eqref{eq:gtv} imply that 
$$ \inner{A^\ast p_{\alpha_n}}{u_{\alpha_n}}_\Ltwou = \TV{u_{\alpha_n}}.$$ 
Since $\TV{\cdot}$ is lower semi-continuous on $\L^1_{\text{loc}}(\R^2)$, we have 
$\TV{\udag} \ls \liminf_n \TV{u_{\alpha_n}}.$
On the other hand, one can write
\begin{align*}
\TV{u_{\alpha_n}} &= \inner{A^\ast p_{\alpha_n}}{\udag}_\Ltwou + \inner{A^\ast p_{\alpha_n}}{u_{\alpha_n} - \udag}_\Ltwou \\
&= \inner{A^\ast p_{\alpha_n}}{\udag}_\Ltwou + \inner{p_{\alpha_n}}{Au_{\alpha_n} - f}_\Ltwof,\end{align*}
where the first term of the right hand side converges to $\inner{A^\ast \overline p}{\udag}_\Ltwou$ 
whereas the second term goes to zero because $(p_{\alpha_n})$ is uniformly bounded in $\Ltwou$ and 
because of the strong $\Ltwof$ convergence $Au_{\alpha_n} \to f$. Moreover, $A^\ast \overline p \in \partial \TV{0}$ because 
$\partial \TV{0}$ is weakly closed. Hence $A^\ast \overline p \in \partial \TV{\udag}$, which is the source condition.
\end{proof}

\section{Convergence of Level-Sets}
\label{sec:cvls}
In order to formulate the main result of this paper we need three definitions:

\begin{definition}
 \label{def:hausdist}
 Let $E$ and $F$ two subsets of $\Omega$. We define the Hausdorff distance between $E$ and $F$ to be the quantity
 \begin{equation*}\begin{aligned}\label{eq:hausdist}
d_H(E, F)&=\max\set{\sup_{x \in E} d(x, F),\, \sup_{y \in F} d(y,E)} \\
&= \max\set{\sup_{x \in E}\, \inf_{y \in F} |x - y|,\, \sup_{y \in F}\, \inf_{x \in E} |x - y|}.\end{aligned}\end{equation*}
If $E_n$ is a sequence of subsets of $\Omega$, we say that $E_n$ Hausdorff converges to $F$ whenever $d_H(E_n, F) \to 0.$
\end{definition}

\begin{definition}
Let $u_{\alpha,w}$ denote the minimizer of $\mathcal F_{\alpha}$, $\hat{\mathcal F}_{\alpha}$, respectively, with 
the data $f+w$, where $w$ can be considered as some error, as already considered in Lemma \ref{lem:uconv}. 
For every $t \in \R$, we denote by $ U_{\alpha,w}^{(t)}$ the $t$ level-set of $u_{\alpha,w}$, that is
\begin{align*}
 U_{\alpha,w}^{(t)} &:= \{x \in \Omega \ \mid \ u_{\alpha,w}(x) \gs t\} & \text {for } t\gs 0, \\
 U_{\alpha,w}^{(t)} &:= \{x \in \Omega \ \mid \ u_{\alpha,w}(x) \ls t\} & \text {for } t< 0\;.
\end{align*}
This choice of the level-sets ensures that the volumes of the level-sets are always finite 
(except the zero one that should be considered separately, see \cite{ChaDuvPeyPoo17}).
Moreover, we call $U_\dag^{(t)}$ the level-sets of $u^\dag$.
\end{definition}

\begin{definition} Let $v \in \L^1 (\Omega)$. Then a set $E$ is said to have 
variational curvature $v$ in $\Omega$ if
\begin{itemize}
\item The perimeter in $\Omega$ of $E$ is finite.
\item Let $\mathcal{E}:=\set{ F \subseteq \Omega: F \Delta E \text{ is compactly supported in }\Omega }$. $E$ minimizes the functional 
\begin{equation*}\label{eq:curvfunc}
\mathcal{E} \ni F  \mapsto Per(F)-\int_F v.
\end{equation*}
That means that for every $F \in \mathcal{E}$ 
\begin{equation}\label{eq:mincpct}
\per(E) - \int_E v \ls \per(F) - \int_F v. 
\end{equation}
\end{itemize}
\end{definition}

\begin{remark}
For smooth sets the notion of variational curvature is strongly related to the differential notion of curvature. Indeed, assuming that the boundary of $E$ is smooth and that its variational curvature $v$ is also smooth, one may consider diffeomorphic deformations $\phi^s: \Omega \to \Omega$ applied to $E$ such that each boundary point $x \in \partial E$ is mapped to $x + s h(x) \nu(x)$, where $\nu$ is the outer unit normal vector to $E$
and $h : \Omega \to \R$ is a smooth function. We obtain at $s=0$ \cite[Section 17.3]{Mag12}
\[\frac{d}{ds} \per\left(\phi^s \big(E\big)\right) = \int_{\partial E} h(x)\kappa(x) \,d\mathcal{H}^1(x) \text{ and}\]
\[\frac{d}{ds} \int_{\phi^s \big(E\big)} v = \int_{\partial E} h(x)v(x) \,d\mathcal{H}^1(x),\]
where $\kappa$ is the curvature of $\partial E$ and $\mathcal H^1$ is the $1$-dimensional Hausdorff measure. Since $h$ was arbitrary and using the minimality \eqref{eq:mincpct} of $E$, we must have $\restr{\kappa}{\partial E} =  v$.

In \cite{BarGonTam87}, the authors show that every set with finite perimeter has a variational curvature in $\L^1(\R^2)$, so such a quantity will exist for every set considered in this paper.
\end{remark}

\begin{figure}
\centering
 \includegraphics[width = 0.7\textwidth]{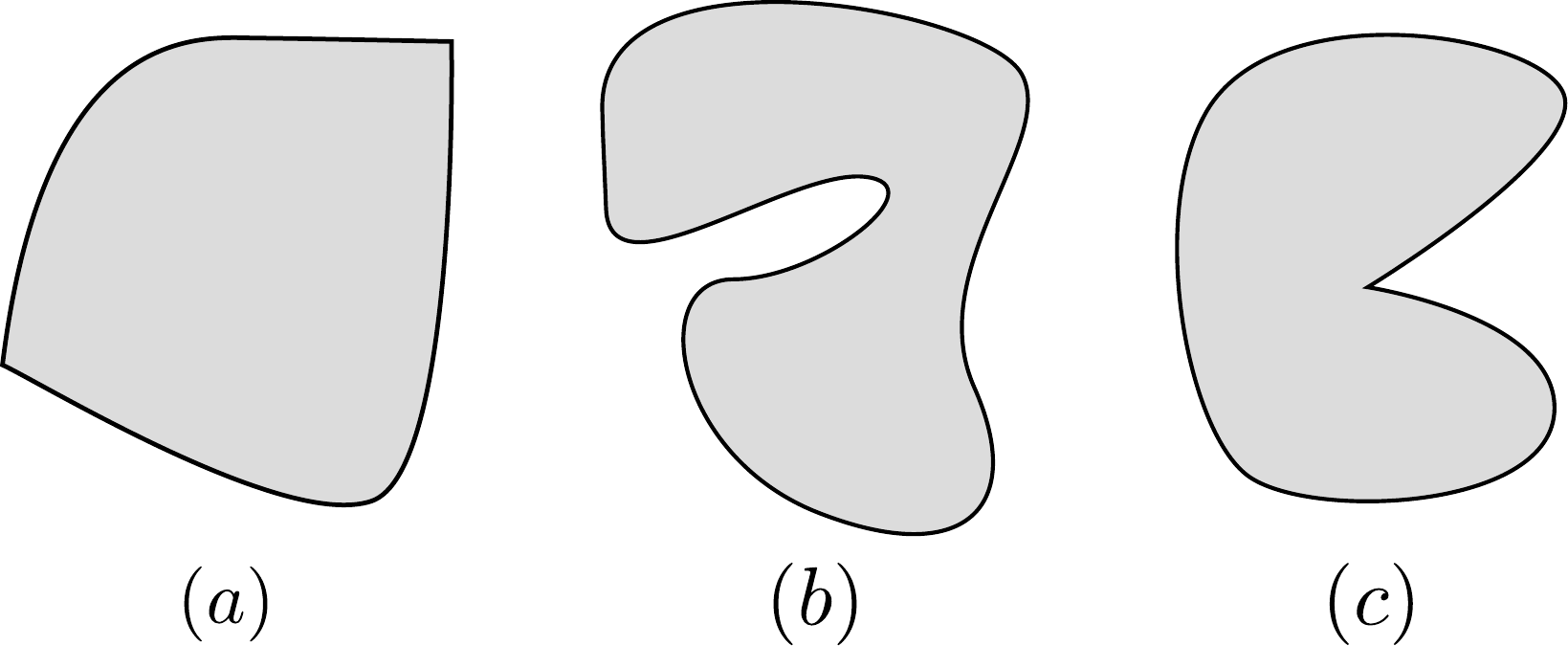}
 \caption{Three Lipschitz domains. Domains $(a)$ and $(b)$ have a variational curvature with lower bound in $\L^2$ of their complements, whereas $(c)$, because of the inside corner, does not.}
 \label{fig:domains}
\end{figure}

Using this definition of a variational curvature we can formulate our main result:
\begin{theorem}\label{th:lsconv}
Assume that either:
\begin{itemize}
\item For Dirichlet boundary conditions or $\Omega = \R^2$, let $(w_n) \subset \Ltwof$ and $\alpha_n \to 0^+$ such that
\begin{equation}\label{eq:paramChoiceDirichlet}
\frac{\| w_n\|_{\Ltwof} \Vert A^\ast \Vert}{\alpha_n} \ls \eta < 2 \sqrt{\pi}.
\end{equation}
If $\Omega$ is bounded, assume further that it admits a variational curvature $\kappa_\Omega \in \L^1(\R^2)$ such that 
$\kappa_\Omega \gs g$ for some $g \in \L^2(\R^2 \setminus \Omega)$. 
\item For Neumann boundary conditions, let $(w_n) \subset \Ltwof$ and $\alpha_n \to 0^+$ such that
\begin{equation}\label{eq:paramChoiceNeumann}
\frac{\| w_n\|_{\Ltwof} \Vert A^\ast \Vert}{\alpha_n} \ls \eta < \frac{1}{C(\Omega)},\end{equation}
with $C(\Omega)$ is some Sobolev-Poincar\'e constant, to be specified later.
\end{itemize}
Let $u_n:=u_{\alpha_n,w_n}$ denote a minimizer of the functional $\mathcal F_{\alpha_n}$, $\hat{\mathcal F}_{\alpha_n} $, respectively, with data $f+w_n$ 
and let $U_n:=U_{\alpha_n,w_n}$ denote the level sets of $u_n$. 
Then up to a subsequence and for almost all $t \in \R$ we have that
\begin{equation}\lim_{n \to \infty} |U_n^{(t)} \Delta U_\dag^{(t)}| = 0, \label{eq:mainth1} \end{equation}
 \begin{equation} \quad \mbox{and} \quad \lim_{n\to \infty} \partial U_n^{(t)} = \partial U_\dag^{(t)}, \label{eq:mainth2} \end{equation}
where the second limit is understood in the sense of Hausdorff convergence.
\end{theorem}
\begin{remark}
 The restriction $\kappa_\Omega \gs g  \in L^2$ that is put on $\Omega$ in Theorem \ref{th:lsconv}  roughly means that inside corners (where the curvature is negative, see Figure \ref{fig:domains} $(c)$) are not allowed. Indeed, corners are known not to have a curvature in $\L^2(\R^2)$ \cite[Theorem 1.1]{GonMasTam93}. However, many interesting domains satisfy $\kappa_\Omega \gs g$ with $g \in \L^2(\R^2 \setminus \Omega)$ (see Figure \ref{fig:domains} $(a)$ and $(b)$), in particular:
\begin{itemize}
\item Every convex domain, even with corners, has a variational curvature $\kappa_\Omega$ such that 
      $\kappa_\Omega = 0$ on $\R^2 \setminus \Omega$, since a convex set minimizes perimeter among outer perturbations. 
      Indeed, if $F \supset E$, this writes $ \per(F) \gs \per(E) $
which is precisely \eqref{eq:mincpct} after using $\int_{F} \kappa_{\Omega} = \int_{E} \kappa_{\Omega}+\int_{F\setminus E} \kappa_{\Omega}.$
\item Any $C^{1,1}$ domain has a curvature $\kappa_\Omega \in \L^\infty(\R^2)$. To see this, first notice that at boundary points of a $C^{1,1}$ set one can place balls of radius bounded below and completely inside or outside $\Omega$ \cite[Theorems 7.8.2 (ii) and 7.7.3]{DelZol11}. Moreover, it is proved in \cite[Remark 1.3 (ii)]{BarGonTam03} that the variational curvature constructed in \cite{GonMasTam93} for a set with the mentioned property is bounded.
\end{itemize}
\end{remark}

\subsection{Proof of Theorem \ref{th:lsconv}}
 The proof is along the lines of \cite{ChaDuvPeyPoo17}, however taking into account that the operator $A$ is not the identity and that we consider various boundary conditions (Dirichlet and Neumann cases). We give its architecture here, postponing the proofs of the main lemmas to the rest of this section as well as Section \ref{sec:density}.

 The proof consists of two steps which correspond to proving \eqref{eq:mainth1} and \eqref{eq:mainth2} respectively.
\begin{proof}[Proof of \eqref{eq:mainth1}]
Let us show first that it is sufficient to prove the strong convergence of $u_n \to u^\dag$ in $\L^1(\R^2).$ Indeed, Fubini's theorem implies
$$ \int_{\Omega} |u_n - u^\dag| = \int_{-\infty}^{+\infty} |U_n^{(t)} - U_\dag^{(t)}| \, dt.$$ Then, the strong convergence of $u_n$ would imply the $\L^1$ convergence of the sequence of functions $$t\mapsto |U_n^{(t)} - U_\dag^{(t)}|,$$ which would imply that one can find a subsequence (not relabelled) such that for almost every $t$,
$$|U_n^{(t)} - U_\dag^{(t)}| \underset{n \to \infty}{\longrightarrow} 0.$$

We now prove the $\L^1$ convergence of $u_n$. Since the conditions \eqref{eq:paramChoiceDirichlet} and \eqref{eq:paramChoiceNeumann} are stronger 
 than \eqref{eq:paramChoiceConv}, we can apply Lemma \ref{lem:uconv} and it follows that $u_n \to u^\dag$
 strongly in $\L^1_{\text{loc}}(\R^2)$ along a subsequence that we do not relabel. To remove the ``loc'', we will prove that all the $u_n$ are actually supported in a same ball. To see this, we first need to investigate geometrically the level-sets of $u_n$ and see that their variational curvature is related to the maximizer of the dual problem \eqref{eq:dual}.
 \begin{lemma}\label{le:lsp}
 Let 
 \begin{equation}
\label{eq:vaw}
 v_{\alpha,w}= A^\ast p_{\alpha,w} = \frac{1}{\alpha}A^\ast(f+w-Au_{\alpha, w}),
\end{equation}
where $p_{\alpha,w}$ is the maximizer of the dual problem \eqref{eq:dual} with data $f+w$ 
(replacing $f$ in \eqref{eq:dual}). Then each level set $U^{(t)}_{\alpha,w}$ of $u_{\alpha, w}$ has a variational curvature $\operatorname{sgn} (t) v_{\alpha,w}$. In addition,  
\begin{equation}
 \label{eq:curvls2}
 \per(U^{(t)}_{\alpha,w}) = 
 \operatorname{sgn}(t) \int_{U^{(t)}_{\alpha,w}} v_{\alpha,w}.
\end{equation}
\end{lemma}
 
First, the $v_{\alpha,w}$ are related to the $v_\alpha := A^\ast p_\alpha$, $p_\alpha$ being the maximizer of the dual functional $\mathcal D_\alpha$.
\begin{lemma}
 \label{lem:paramChoice}
\begin{equation*}\label{eq:dualNoiseEst}
\Vert v_{\alpha} - v_{\alpha,w} \Vert_{\Ltwou} \ls \frac{ \Vert w \Vert_{\Ltwof} \Vert A^\ast \Vert}{\alpha} \ls \eta,
\end{equation*}
 where $\eta$ is defined in \eqref{eq:paramChoiceDirichlet}.
\end{lemma}

Moreover, from Prop. \ref{prop:conv} and the boundedness of $A$ it follows that since the source condition \eqref{eq:srccnd} holds, 
 \begin{equation*}
  \label{eq:equi}
 (\va):=(A^\ast\pa) \overset{\Ltwou}{\longrightarrow} v_0 := A^\ast p_0
 \end{equation*}
 and the family $(v_\alpha)$ is therefore equiintegrable (see Definition \ref{def:equiint} in the appendix).

In the rest of the proof, we see that the fact that the level-sets $U_{\alpha,w}^{(t)}$ have variational curvatures $v_{\alpha,w}$ close to the equiintegrable family $v_\alpha$ implies some uniform regularity. The necessity of the restrictions for $\eta$ of \eqref{eq:paramChoiceDirichlet} and  \eqref{eq:paramChoiceNeumann} is made apparent in Sections \ref{sec:cpctspt} and \ref{sec:density} below. First,

\begin{lemma}\label{le:compact} Assume \eqref{eq:paramChoiceDirichlet}. Then, the elements of
\begin{equation}\label{eq:curvls2-2}
\mathcal{E}:=\set{E \subset \Omega \, \middle | \, \per(E)=\int_E v_{\alpha, w} \text{ with } v_{\alpha,w} \text{ from }\eqref{eq:vaw}},
\end{equation}
have the following properties:
\begin{enumerate}
 \item There exists a constant $C > 0$ such that for all $E \in \mathcal{E}$, $\per(E) \leq C$,
 \item There exists a constant $R > 0$ such that for all $E \in \mathcal{E}$, $E \subseteq \mathcal{B}(0,R)$.
\end{enumerate}
\end{lemma}
 This lemma implies that the level-sets of $u_n$, which belong to $\mathcal E$, are contained in some $B(0,R).$ That means that the $u_n$ are all supported in a common ball. From their $\L^1_{\text{loc}}$ convergence, we then deduce the full $\L^1$ convergence of $u_n$ to $u^\dag$ and therefore \eqref{eq:mainth1}. 
\end{proof}   

\begin{proof}[Proof of \eqref{eq:mainth2}]
We will deduce \eqref{eq:mainth2} from \eqref{eq:mainth1}. This step is performed through stronger regularity for the level-sets $U_n$. In our case, the adequate property is termed \emph{weak-regularity} in \cite{ChaDuvPeyPoo17}, and relates to the well-known density 
estimates for $\Lambda$-minimizers of perimeter \cite[Theorem 21.11]{Mag12}. This is the main ingredient of the proof of Theorem \ref{th:lsconv}. We can state it as
\begin{lemma}
\label{lem:densest}
Under the assumptions of Theorem \ref{th:lsconv} (in particular \eqref{eq:paramChoiceDirichlet}, \eqref{eq:paramChoiceNeumann} resp.), the level-sets $U_{\alpha,w}^{(t)}$ satisfy the following property. There exists $C >0$  and $r_0 >0$ (which do not depend on $\alpha$, $w$ or $t$ provided \eqref{eq:paramChoiceDirichlet} (resp. \eqref{eq:paramChoiceNeumann}) is satisfied) such that for every $\alpha$, $w$ satisfying \eqref{eq:paramChoiceDirichlet} (resp. \eqref{eq:paramChoiceNeumann}) and every $r < r_0$ and $x \in \partial U_{\alpha,w}^{(t)}$, one has 
\begin{equation}\label{eq:densityest}
\frac{|B(x,r) \cap U_{\alpha,w}^{(t)}|}{|B(x,r)|}\gs C \text{ and }\frac{|B(x,r) \setminus U_{\alpha,w}^{(t)}|}{|B(x,r)|}\gs C.\end{equation}
\end{lemma}
We prove this property for all the different boundary conditions in Section \ref{sec:density}. To conclude the proof of \eqref{eq:mainth2}, one need the
\begin{lemma}
 \label{lem:l1haus}
Let $E_n, F$ measurable subsets of $\Omega$ such that $E_n \overset{\L^1(\Omega)}{\longrightarrow} F$, such that there exists $R>0$ independent of $n$ with $E_n \subset B(0,R)$ and such that the conclusion of Lemma \ref{lem:densest} holds (with $U_{\alpha,w}^{(t)}$ replaced by $E_n$). Then, the convergence holds in Hausdorff sense, that is the Hausdorff distance (Definition \ref{def:hausdist})
$$ d_H(E_n, F) \to 0.$$
\end{lemma}
Applying this lemma to $U_n^{(t)}$ for a $t$ such that \eqref{eq:mainth1} holds, we conclude the proof of \eqref{eq:mainth2}.
\end{proof}

\begin{proof}[Proof of Lemma \ref{lem:l1haus}]
Let us consider the definition of Hausdorff distance:
\begin{equation*}d_H(E_n, F)= \max\set{\sup_{x \in E_n}\, \inf_{y \in F} |x - y|,\, \sup_{y \in F}\, \inf_{x \in E_n} |x - y|},\end{equation*}
and suppose without loss of generality that the first term of the right hand side does not converge to 0. This would imply that there is a $\delta > 0$ (we can take $\delta < r_0$) and $x_n \in E_n$ such that $d(x_n, F)\gs \delta$, and in particular $B(x_n, \delta) \cap F = \emptyset$. This implies using the density estimate \eqref{eq:densityest} that
\[\left|E_n \,\Delta\, F \right| \gs \left| (E_n \cap B(x_n, \delta)) \setminus F \right| = \left| E_n \cap B(x_n, \delta) \right| \gs C\delta^2,\]
contradicting the $\L^1$ convergence. 
\end{proof}

\subsection{The level-sets have prescribed curvature: proof of Lemma \ref{le:lsp}}
\begin{proof}[Proof of Lemma \ref{le:lsp}]
It is proved in \cite[Proposition 3]{ChaDuvPeyPoo17} by slicing equations \eqref{eq:extcond1} and \eqref{eq:gtv} and using the coarea and layer cake formulas \eqref{eq:coarea} and \eqref{eq:layercake} that the extremality relation \eqref{eq:extcond1} is equivalent to the statement that for every $F \subset \Omega$ and every $t \neq 0$
\begin{equation}
 \per(F) - \operatorname{sgn}(t) \int_F v_{\alpha,w} \gs \per(U^{(t)}_{\alpha,w}) - 
 \operatorname{sgn}(t) \int_{U^{(t)}_{\alpha,w}} v_{\alpha,w},
 \label{eq:curvls}
\end{equation}
which implies that $U^{(t)}_{\alpha,w}$ has a variational curvature $ \operatorname{sgn} (t) v_{\alpha,w}$. Furthermore, it is also shown in \cite[Proposition 3]{ChaDuvPeyPoo17} that \eqref{eq:gtv} implies that the $U^{(t)}_{\alpha,w}$ satisfy
\begin{equation*}
 \per(U^{(t)}_{\alpha,w}) = 
 \operatorname{sgn}(t) \int_{U^{(t)}_{\alpha,w}} v_{\alpha,w}.
\end{equation*}\end{proof}

\subsection{Parameter choice: proof of Lemma \ref{lem:paramChoice}}
\begin{proof}[Proof of Lemma \ref{lem:paramChoice}]
From \eqref{eq:dual} with $f$ replaced by $f+w$ it follows that $p_{\alpha,w}$ is the $\Ltwof$ orthogonal projection of $\frac{f+w}{\alpha}$ onto the convex set 
$$\set{ p \in \Ltwof \, \ \middle \vert\  A^\ast p \in \partial \TV{0}}.$$ 
The non-expansiveness of the projection operator leads to
$$\Vert p_{\alpha} - p_{\alpha,w} \Vert_{\Ltwof} \ls \frac{ \Vert w \Vert_{\Ltwof}}{\alpha},$$
which, together with the boundedness of $A^\ast$, means that 
\begin{equation}
\Vert v_{\alpha} - v_{\alpha,w} \Vert_{\Ltwou} \ls \frac{ \Vert w \Vert_{\Ltwof} \Vert A^\ast \Vert}{\alpha} \ls \eta,
\label{eq:closetonoiseless}
\end{equation}
where $\eta$ is defined in \eqref{eq:paramChoiceDirichlet}.
\end{proof}

\subsection{Upper bounds and compact support: proof of Lemma \ref{le:compact}}
\label{sec:cpctspt}

\begin{proof}[Proof of Lemma \ref{le:compact}]
We distinguish between the following cases:
\begin{itemize}
\item Let $\Omega$ be bounded, a bound on the perimeter follows easily from \eqref{eq:curvls2} and \eqref{eq:paramChoiceDirichlet} (resp. \eqref{eq:paramChoiceNeumann}):
\begin{equation*}\label{eq:dirichletupperbound}
\begin{aligned}
 \per(E) \ls \left \vert \int_E (v_{\alpha,w} - v_{\alpha}) \right \vert + \left \vert \int_E v_{\alpha} \right \vert &\ls \eta \sqrt{|\Omega|} + \sqrt{|\Omega|}\Vert v_{\alpha} \Vert_{\Ltwou} \\
 & \ls \left( \eta + \sup_{\alpha} \Vert v_{\alpha} \Vert_{\Ltwou} \right) \sqrt{|\Omega|}.
\end{aligned}
\end{equation*}

\item If $\Omega = \R^2$, then the proof is very similar to what is done in \cite{ChaDuvPeyPoo17}. 
We sketch now the arguments given in \cite{ChaDuvPeyPoo17}, that apply directly to this case. 

Here, by Prop. \ref{prop:conv}, we have that $v_{\alpha} \to v_0$ strongly in $\Ltwou$, and therefore the family $(v_{\alpha})$ is $\mathrm{L}^2$-equiintegrable, which in particular means that for every $\epsilon > 0$, one can find a ball $B(0,R)$ such that 
$$ \int_{\R^2 \setminus B(0,R)} v_{\alpha}^2 \ls \epsilon.$$
Then, for every $E$ with finite mass that satisfies \eqref{eq:curvls2-2} and provided $\alpha$ and $w$ satisfy
\eqref{eq:paramChoiceDirichlet},
\begin{align*}
 \per(E) &\ls \left \vert \int_E (v_{\alpha,w} - v_{\alpha}) \right \vert + \left \vert \int_{E \cap B(0,R)} v_{\alpha} \right \vert + \left \vert \int_{E \setminus B(0,R)} v_{\alpha} \right \vert \\
 & \ls \eta \sqrt{|E|} + \sqrt{|B(0,R)|}\Vert v_{\alpha} \Vert_{\Ltwou} + \sqrt{|E \setminus B(0,R)|} \epsilon \\
 & \ls \left( \eta + \sup_{\alpha} \Vert v_{\alpha} \Vert_{\Ltwou} \right) \sqrt{|B(0,R)|} + (\eta+\epsilon) \sqrt{|E\setminus B(0,R)|}.
\end{align*}
Now, the isoperimetric inequality (that is, $4\pi \abs{E} \leq \per(E)^2$) and sub-additivity of the perimeter lead to
$$ \sqrt{|E \setminus B(0,R)|} \ls \frac{1}{\sqrt{4\pi}}\per (E\setminus B(0,R)) \ls \frac{1}{\sqrt{4\pi}} \big( \per(E) + \per(B(0,R)) \big),$$
which when used in the previous equation, since $\epsilon$ is arbitrary and $\eta < 2 \sqrt{\pi}$, implies that $ \per(E)$ is bounded uniformly in $\alpha$. Once again using the isoperimetric inequality yields the boundedness of $|E|$ independently of $\alpha$, as long as \eqref{eq:paramChoiceDirichlet} is satisfied.
\end{itemize}

We now prove that the mass and perimeter of level-sets of $u_{\alpha,w}$ are bounded away from zero. The equiintegrability of $(v_{\alpha})$ ensures that there is no concentration of mass for $v_{\alpha}$: 
$\int_E v_{\alpha}^2$ is small if $|E|$ is small. Then, if $E$ satisfies \eqref{eq:curvls2-2}, Cauchy Schwarz inequality provides an inequality of the type
$$ \per(E) \ls \eps \sqrt{|E|},$$
which together with the isoperimetric inequality, implies $\per(E) \ls C \eps \per(E)$, which is not possible for $\eps$ too small. Therefore, $|E|$ must be bounded away from zero (and $\per(E)$ as well thanks to the isoperimetric inequality).

It is shown in \cite[Remark 4]{ChaDuvPeyPoo17} (using \cite[Corollary 1]{AmbCasMasMor01}) that if $E$ has a finite mass and satisfies \eqref{eq:curvls2-2}, $E$ can be split into connected components which also satisfy \eqref{eq:curvls2-2}. Therefore, the perimeter and mass of such components are bounded from above and below, which implies that there can only be finitely many of them. Since their perimeter is bounded, their diameter is bounded too, which implies that they all lie in a ball $B(0,R).$ So does $E$.
\end{proof}
\begin{remark}
As a byproduct of the previous proof, one can notice that all level-sets of $u^\dag$ belong to some ball $B(0,R)$, which means that $\partial \TV{u^\dag} \neq \emptyset$ implies that $u^\dag$ has a compact support. To our knowledge, this property was never stated before, although it is implicit in \cite{ChaDuvPeyPoo17}. Since it is a result on the subgradient, it applies whether $A = \Id$ or not.
\end{remark}

\section{Proof of the density estimates: proof of Lemma \ref{lem:densest}}\label{sec:density}
In this section, we prove Lemma \ref{lem:densest}, that is we derive the density estimates \eqref{eq:densityest} in each of the three boundary frameworks that are mentioned in this article. The proof follows the usual strategy for this kind of estimates (see \cite{Mag12}, for example), but the appearance of different boundary conditions requires a closer examination.

The general strategy of the proof is to use minimality of a set in problem \eqref{eq:curvls} and compare it with the sets obtaining by adjoining or substracting pieces of balls centered at a point of its boundary, leading to the first and second parts of \eqref{eq:densityest} respectively.

In what follows we consider only the first estimate, since the second one can be derived analogously. We emphasize that the bounds obtained need to be uniform in $\alpha$ in order to obtain the desired convergence.

Let us first prove the
\begin{lemma}
 Let $\kappa \in \L^2(\Omega)$ (with $\Omega$ bounded or $\Omega = \R^2$) and $E \subset \Omega$ minimize 
 $$ F \mapsto \per(F \, ;\, \Omega) - \int_F \kappa,$$
 among $F \subset \Omega$ Lebesgue measurable.
 Then, one has for almost every $r$
\begin{equation} \per(E \cap \Br) - \int_{E\cap \Br} \kappa \ls 2 \per(\Br\, ;\, E^{(1)}), \label{eq:compareGen} \end{equation}
where $E^{(1)}$ denotes the set of points of density $1$ in $E$ (see Definitions \ref{def:densitypoints} and \ref{def:relativeper} in the appendix).
\end{lemma}
\begin{remark}
One can note that 
\begin{equation}\label{eq:perH1}\per(\Br\, ;\, E^{(1)}) = \mathcal H^1(\partial \Br \cap E) \text{ for almost every $r$,}\end{equation}
for $\mathcal H^1$ the $1$-dimensional Hausdorff measure. In fact, \eqref{eq:compareGen} can be proved for all $r$ by keeping track of extra terms in  \eqref{eq:perdiff1} and \eqref{eq:perdiff2} that appear when the sets $E$ and $\Br$ have tangential contact.
\end{remark}

\begin{proof}We use the following inequality, valid for every finite perimeter sets $F, G \subset \Omega$,
\begin{equation}\label{eq:perdiff1}
\per(F \setminus G ) + \per(G \setminus F) \ls \per(F) + \per(G),
\end{equation}
which can be proved by using \eqref{eq:1611} twice. We will also apply the following equality, which holds for almost every $r > 0$
\begin{equation}\label{eq:perdiff2}
\begin{aligned}
&\per((\Br \cap \Omega) \setminus E)  \\
&\quad = \per(\Br \cap \Omega \, ;\, E^{(0)} \cap \Omega) + \per(E \, ;\, \Br \cap \Omega) \\
&\quad = \per(\Br \cap \Omega \, ;\, E^{(0)} \cap \Omega) + \per(E\cap \Br \, ;\, \Omega) - \per(\Br \cap \Omega \, ;\, E^{(1)} \cap \Omega).\\
\end{aligned}
\end{equation}
This can be deduced using \eqref{eq:1611} in the first equality and \eqref{eq:1610} in the second one and the fact that the set of tangential contact 
$$\{ \nu_E = \pm \nu_{\Br}\}$$
is contained in the set
$$ \partial (\Br \cap \Omega) \cap \partial^* E$$ and that for almost every $r$,
\begin{equation}\mathcal{H}^1\left[  \left(\partial (\Br \cap \Omega) \cap \partial^* E \right) \cap \Omega \right]=0,\label{eq:zeroae} \end{equation}  
since $\mathcal{H}^{1}(\partial^* E) < \infty$, where $\partial^* E$ is the reduced boundary of $E$ (see Appendix).

Using the minimality of $E$, then the two formulas above, and the additivity of perimeter, we get
\begin{equation*}\begin{aligned}
\per(E)&-\int_E \kappa \ls \per(E \setminus \Br ) - \int_{E \setminus \Br} \kappa \\
&\ls \per(E) + \per(\Br \cap \Omega) - \per((\Br \cap \Omega) \setminus E) - \int_{E \setminus \Br} \kappa \\
&=\per(E) + \per(\Br \cap \Omega) - \per(\Br \cap E) - \per(\Br \cap \Omega \, ;\, E^{(0)} \cap  \Omega) \\
&\quad+ \per(\Br \cap \Omega \, ;\, E^{(1)} \cap \Omega) - \int_{E \setminus \Br} \kappa \\
& = \per(E) -  \per(\Br \cap E) + 2 \per(\Br\, ;\, E^{(1)}\cap \Omega) - \int_{E \setminus \Br} \kappa ,
\end{aligned}\end{equation*}
where in the last equality we use Theorem \ref{th:federer} and again \eqref{eq:zeroae} to see that for a.e. $r$,
$$\per(\Br \cap \Omega) = \per(\Br \cap \Omega \, ;\, E^{(1)} \cap \Omega) + \per(\Br \cap \Omega \, ;\, E^{(0)} \cap \Omega ). $$
 Since $E^{(1)} \subset \Omega$, the inequality above is the statement of \eqref{eq:compareGen}.
\end{proof}

\subsection{The \texorpdfstring{$\R^2$}{R2} case.}
\label{sec:densr2}
Here, $\Omega = \R^2$ and the proof is then the one presented in \cite{ChaDuvPeyPoo17} up to making more explicit the constants involved. Let us consider a level-set $U_{\alpha,w}^{(t)}$ (we assume without loss of generality that $t >0$) of $u_{\alpha,w}$ that therefore minimizes
$$F \mapsto \per(F) - \int_F v_{\alpha,w},$$
and $x \in \partial U_{\alpha,w}^{(t)}$.
Thanks to the equiintegrability of $v_{\alpha}$ (which, as noted before, follows from the strong convergence in $\L^2$ showed in Proposition \ref{prop:conv}), for every $\epsilon >0$ and $|F| \ls \pi r_0^2$ with $r_0$ small enough (independent of $\alpha$ but dependent of $\epsilon$) one has

\begin{equation} \label{eq:equiint} \left(\int_F |v_{\alpha}|^2\right)^{1/2} \ls \epsilon.\end{equation}

Then, \eqref{eq:paramChoiceDirichlet}, \eqref{eq:closetonoiseless} and the above imply that for $r \ls r_0$,
\begin{align*}\left \vert \int_{U_{\alpha,w}^{(t)} \cap \Br} v_{\alpha,w} \right \vert & \ls |U_{\alpha,w}^{(t)} \cap \Br|^{1/2} \Vert v_{\alpha,w} \Vert_{\L^2(\Br)} & \\
&\ls |U_{\alpha,w}^{(t)} \cap \Br|^{1/2} \left( \Vert v_{\alpha} \Vert_{\L^2(\Br)} + \eta \right) \\
& \ls |U_{\alpha,w}^{(t)} \cap \Br|^{1/2} (\epsilon + \eta).
\end{align*}
Using the above in \eqref{eq:compareGen} (with $\kappa = v_{\alpha,w}$), we obtain
$$\per(U_{\alpha,w}^{(t)} \cap \Br) - |U_{\alpha,w}^{(t)} \cap \Br|^{1/2} (\epsilon + \eta) \ls 2 \per(\Br \, ; \, {U_{\alpha,w}^{(t)}}^{(1)}),$$
which combined with the isoperimetric inequality in $\R^2$ and finally \eqref{eq:perH1} yields
\begin{equation} |U_{\alpha,w}^{(t)} \cap \Br|^{1/2} (2\sqrt{\pi} - \epsilon - \eta) \ls 2 \mathcal H^1(U_{\alpha,w}^{(t)} \cap \partial \Br). \label{eq:diffeq} \end{equation}

Now, denoting by
$$ g(r) := |U_{\alpha,w}^{(t)} \cap \Br|,$$
the coarea formula \eqref{eq:coarea} for the distance to $x$ implies for a.e. $r$, $g'(r) = \mathcal H^1(U_{\alpha,w}^{(t)} \cap \partial \Br).$
As a result, \eqref{eq:diffeq} reads 
\begin{equation*}\label{eq:diffeqg}(2\sqrt{\pi} - \epsilon - \eta) \sqrt{g} \ls 2 g'.\end{equation*} 
Now, if $\eta$ and $\epsilon$ are chosen such that $\epsilon + \eta < 2 \sqrt{\pi}$, one can integrate on both sides and use $g(0)= 0$ to get $(2\sqrt{\pi} - \epsilon - \eta)r \ls 4 \sqrt{g(r)}$, which reads
$$\frac{|B(x,r) \cap U_{\alpha,w}^{(t)}|}{|B(x,r)|} \gs \frac{(2\sqrt{\pi} - \epsilon - \eta)^2 r^2}{16 \pi r^2}  = \frac{(2\sqrt{\pi} - \epsilon - \eta)^2}{16\pi},$$
where the right hand side is uniform in $\alpha$. Since $\epsilon$ was arbitrary and the parameter choice \eqref{eq:paramChoiceDirichlet} implies $\eta < 2 \sqrt{\pi}$, we obtain \eqref{eq:densityest}.

\subsection{The Dirichlet case}
In this subsection, we consider the case of Dirichlet conditions in a bounded domain, and see that it can be treated through a variational problem  formulated in $\R^2$:

\begin{lemma}Assume that $\Omega$ admits a variational curvature $\kappa_\Omega$ such that $\kappa_\Omega \gs g$ with $g \in \L^2(\R^2)$, and let $E \subset \Omega$ satisfy \eqref{eq:curvls} (we assume that $t >0$). Then, $E$ satisfies the following variational problem among sets $F\subset \R^2$ such that $|F \Delta E|$ is bounded:
\begin{equation}\label{eq:globalcurv}\per(E) - \int_E \kaw \ls \per(F) - \int_F \kaw,\quad
\text{where }\kaw=v_{\alpha, w}1_\Omega + g 1_{\R^2 \setminus \Omega}.\end{equation} 
\end{lemma}
\begin{proof}Similarly to \cite[Lemma, p.~132]{BarMas82}, we consider the constraint as an obstacle and we write (comparing $E$ and $F \cap \Omega$ in \eqref{eq:globalcurv})
\begin{equation}\label{eq:inside}\begin{aligned}\per(E) - \int_E \kaw &= \per(E) - \int_E v_{\alpha, w} \ls \per(F \cap \Omega) - \int_{F \cap \Omega} v_{\alpha, w} \\ 
&\ls \per(F)+\per(\Omega)-\per(F \cup \Omega) - \int_{F \cap \Omega} v_{\alpha, w} \\
&= \per(F)+\per(\Omega)-\per(F \cup \Omega) - \int_{F \cap \Omega} \kaw.\end{aligned}\end{equation}
On the other hand, the variational curvature of $\Omega$ implies
\begin{equation*}\label{eq:outside}\begin{aligned}\per(\Omega) - \int_\Omega \kappa_\Omega &\ls \per(F \cup \Omega) - \int_{F \cup \Omega} \kappa_\Omega,\end{aligned}\end{equation*}
which we can use in \eqref{eq:inside} along with the definition of $\kappa_{\alpha,w}$ to get
\begin{equation*}\begin{aligned}
\per(E) - \int_E \kaw &\ls \per(F) - \int_{F \cap \Omega} \kaw - \int_{F \setminus \Omega} \kappa_\Omega \\
&\ls \per(F) - \int_{F \cap \Omega} \kaw - \int_{F \setminus \Omega} g \\
&= \per(F) - \int_{F \cap \Omega} \kaw - \int_{F \setminus \Omega} \kaw \\
&= \per(F) - \int_F \kaw.
\end{aligned}
\end{equation*}
\end{proof}
Now, for $E = U_{\alpha,w}^{(t)}$ a level-set of $u_{\alpha,w}$, $x \in \partial U_{\alpha,w}^{(t)}$, we may perturb $U_{\alpha,w}^{(t)}$ with balls $B(x,r)$ not necessarily contained in $\Omega$. Using the definition of the $(\kappa_{\alpha,w})$ we get that $(\kappa_{\alpha,0})$ is equiintegrable, and using also \eqref{eq:closetonoiseless}, that $\|\kappa_{\alpha,w}-\kappa_{\alpha,0}\| \ls \eta$. Therefore, we can apply the $\R^2$ density estimates of Section \ref{sec:densr2} to obtain \eqref{eq:densityest} for the Dirichlet boundary conditions.

\subsection{The Neumann case}
For $\per(\cdot \, ; \, \Omega)$, the usual isoperimetric inequality does not hold. However, since we have assumed that $\Omega$ is such that its boundary can be locally represented as the graph of a Lipschitz function (so it is in particular an extension domain, see \cite[Definition 3.20, Proposition 3.21]{AmbFusPal00}), we can use the following Poincar\'{e}-Sobolev inequality \cite[Remark 3.50]{AmbFusPal00} valid for $u \in \BV(\Omega)$:
\[\norm{u - \frac{1}{|\Omega|}\int_\Omega u}_\Ltwou \ls C(\Omega)\,\TV{u\, ;\, \Omega}.\]

With $u = 1_F$ the characteristic function of some $F \subset \Omega$, the left hand side reads
$$ \int_{\Omega} \left \vert 1_F - \frac{|F|}{|\Omega|} \right \vert^2 = |F| \left( \frac{|\Omega \setminus F|}{|\Omega|} \right)^2 + |\Omega \setminus F| \left( \frac{|F|}{|\Omega|} \right)^2$$
and the inequality reads
\begin{equation}  C(\Omega) \per(F) \gs \left( \frac{|F| \ |\Omega \setminus F|}{|\Omega|^2} \right)^{1/2} \left( |\Omega \setminus F| + |F| \right)^{1/2} \gs \left( \frac{|F| \ |\Omega \setminus F|}{|\Omega|} \right)^{1/2}. \label{eq:PoincareSet} \end{equation}

As before, let $U_{\alpha,w}^{(t)}$ be a level-set of $u_{\alpha,w}$ (that satisfies \eqref{eq:curvls}). Applying \eqref{eq:PoincareSet} to $U_{\alpha,w}^{(t)} \cap \Br$, we get
\begin{equation}\label{eq:isopNeum}\begin{aligned}|U_{\alpha,w}^{(t)} \cap \Br|^{1/2} & \ls C(\Omega) \left( \frac{|\Omega|}{|\Omega \setminus (U_{\alpha,w}^{(t)} \cap \Br)|}  \right)^{1/2} \per( U_{\alpha,w}^{(t)} \cap \Br) \\
& \ls  C(\Omega) \left( \frac{|\Omega|}{|\Omega \setminus \Br|}  \right)^{1/2} \per( U_{\alpha,w}^{(t)} \cap \Br).
\end{aligned}\end{equation}
Now, the parameter choice \eqref{eq:paramChoiceNeumann} implies that one can choose $r_0$ independent of $x$ such that for every $r\ls r_0$, 
$$ \eta <  \frac{1}{C(\Omega)} \left( \frac{|\Omega| -|B(r)|}{|\Omega|}  \right)^{1/2} $$
and such that \eqref{eq:equiint} holds for some $\epsilon$ that satisfies 
$$ \frac{1}{C(\Omega)} \left( \frac{|\Omega| -|B(r_0)|}{|\Omega|}  \right)^{1/2} - \epsilon - \eta > 0.$$

We can then use \eqref{eq:isopNeum} instead of the isoperimetric inequality in \eqref{eq:compareGen} (which holds since $U_{\alpha,w}^{(t)}$ satisfies \eqref{eq:curvls}) and perform the same proof as in Section \ref{sec:densr2} to get the estimate
$$\frac{|B(x,r) \cap U_{\alpha,w}^{(t)}|}{|B(x,r)|} \gs \frac{\left( \frac{1}{C(\Omega)} \left( \frac{|\Omega| -|B(r_0)|}{|\Omega|}  \right)^{1/2} - \epsilon - \eta \right)^2 }{16 \pi },$$
where the right hand side is uniform in $r$ and $x$. This is the first estimate of \eqref{eq:densityest}. In this case, the second estimate of \eqref{eq:densityest} reads
$$\frac{|(B(x,r) \cap \Omega)\setminus U_{\alpha,w}^{(t)}|}{|B(x,r)|} \gs C,$$
which is still enough for the Hausdorff convergence of $\partial U_{\alpha,w}^{(t)}$ for the proof of \eqref{eq:mainth2}.

\section{Examples and discussion}\label{sec:examples}
First, we consider two particular applications. In the first, we check that Theorem \ref{th:lsconv} applies to the inversion of the circular Radon transform, which is used as a model for image formation in synthetic aperture radar \cite{HelAnd87}. In such an application, object detection is often important, and Hausdorff convergence of level-sets corresponds to a kind of uniform convergence of the detected objects. In the second, we numerically confirm the convergence of level sets in the deblurring of a characteristic function, where the theorem applies and the convergence has visual meaning.

We then conclude with some remarks about the differences in the solutions when imposing different boundary conditions, and an accompanying numerical example.

\subsection{The circular Radon transform}
We review from \cite{SchGraGroHalLen09} the problem of inverting the \emph{circular Radon transform} 
\begin{equation}\label{eq:ip:rc0}
    \Rcirc u = v
\end{equation}
in a stable way, where
\begin{equation}\label{eq:ip:rc}
\begin{aligned}
\Rcirc: \L^2(\R^2)  &\to 
        \L^2 \bigl( \Sigma = \mathbb{S}^1 \times (0, 2)\bigr)\,,\\
     u & \mapsto (\Rcirc u) ( \vec{z}, t)
     :=
     t
     \int_{\mathbb{S}^1}
     u(\vec{z}  + t{\vec{\omega}} ) \, d{\cal H}^1({\vec{\omega}})
     \,.
\end{aligned}
\end{equation}
In the following let $\Omega := B(0,1)$ be an open Ball of Radius $1$ with center $0$ in $\R^2$ and 
let $\eps \in (0,1)$. We are considering the spherical Radon transform defined on the 
subspace of functions supported in $\overline{B(0,1-\eps)}$, that is on 
\begin{equation*}
 \L^2(B(0,1-\eps)):= \set{u \in \L^2(\R^2) \ \middle\vert\ \supp(u) \subseteq \overline{B(0,1-\eps)}}.
\end{equation*}

Some properties \cite[Propositions 3.80 and 3.81]{SchGraGroHalLen09} of the circular Radon transform are:
\begin{itemize}
 \item The circular Radon transform, as defined in \eqref{eq:ip:rc},
is well-defined, bounded, and satisfies $\norm{\Rcirc} \leq 2\pi$.
 \item There exists a constant $C_\eps>0$, such that
\begin{equation*}\label{eq:ip:sobolev:help1}
C_\eps^{-1} \norm{\Rcirc u}_2 \leq \norm{i^* (u)}_{1/2,2} \leq C_\eps
\norm{\Rcirc u}_{2} \,, \quad u \in \L^2(B(0,1-\eps)) \,,
\end{equation*}
where $i^*$ is the adjoint of the embedding $i : \W^{1/2,2}(B(0,1)) \to \L^2(B(0,1))$ of the standard 
Sobolev space of differentiation of order $1/2$ on $\Omega$. 
\item For every  $\eps \in (0,1)$ we have

\begin{equation*}
\W^{1/2,2}(B(0,1-\eps)) \subset \range{\Rcircadj} \cap \L^2(B(0,1-\eps)) \,,
\end{equation*}
where 
\begin{multline*}
\W^{1/2,2}(B(0,1-\eps)) \\
:= \set{ u \in \L^2(\R^2) \ \middle\vert\ \supp{u} \subset \overline{B(0,1-\eps)} \text{ and } u|_{B(0,1)} \in \W^{1/2,2}(B(0,1))}.
\end{multline*}
Note that $\W^{1/2,2}$ is not the standard definition of a Sobolev space because we associate with each function 
of the space $\W^{1/2,2}(B(0,1-\eps))$ an extension to $\R^2$ by $0$ outside.
We could also say, in the terminology of this paper, that these functions satisfy zero Dirichlet boundary condition on $B(0,1-\eps)$.

\end{itemize}
It was shown in \cite[Propositions 3.82 and 3.83]{SchGraGroHalLen09} that minimization of the functional \eqref{eq:primalDirichlet} with $A=\Rcirc$:
\begin{itemize}
 \item is well-posed, stable, and convergent.
 \item Moreover, the following result holds: Let $\eps \in (0,1)$ and $\udag$ be the solution of \eqref{eq:ip:rc0}. 
       Then we have the following convergence rates result for TV-regularization: 
       If $\xi \in \partial \TV{\udag} \cap \W^{1/2,2}(B(0,1-\eps))$, then
       \begin{equation*}\label{eq:ip:rate-tv-radon}
        \TV{u_{\alpha(\delta)}^\delta} - \TV{\udag} - \scal{\xi}{u_{\alpha(\delta)}^\delta - \udag} 
        =
        \mathcal{O}(\delta)
        \qquad \text{ for } \alpha(\delta) \sim \delta .
        \end{equation*}
        In the last equation, the left hand side is called \emph{Bregman distance} of $TV$ at $\udag$ and $\xi$.
        
        With the results of this paper, if the parameter $\alpha$ is chosen finer, meaning satisfying \eqref{eq:paramChoiceDirichlet}, we not only 
        have convergence rates of the Bregman distance, but also convergence of the level-sets.
        
        There are particular examples for which the source condition is satisfied:
        \begin{itemize}
        \item Let $\rho \in C^{\infty}_0(\R^2)$ be an adequate mollifier and $\rho_\mu$ the scaled function of $\rho$.
              Moreover, let $x_0 = (0.2,0)$, $a = 0.1$, and $\mu = 0.3$.
              Then
              $$u^\dagger := 1_{B(x_0, a + \mu)} \ast \rho_{\mu}$$ 
              satisfies the source condition. 
        \item Let  $\udag := 1_F$ be the characteristic function of a bounded subset of $\R^2$ with smooth boundary. Then, the source condition is satisfied as well \cite[Example 3.74]{SchGraGroHalLen09}.
        \end{itemize}
\end{itemize}

\subsection{A numerical deblurring example}
The second situation we consider is a numerical deconvolution example, in which a characteristic function has been blurred with a Gaussian kernel and subsequently corrupted by additive Gaussian noise. Both the convolution kernel and the variance of the noise are assumed known, and Dirichlet boundary conditions on a rectangle are used. These choices lead directly to the minimization of \eqref{eq:primalDirichlet} and enable the use of a parameter choice according to \eqref{eq:paramChoiceDirichlet}, so that the the results of Section \ref{sec:cvls} provide convergence of level-lines.

The discretization of choice is the `upwind' finite difference scheme of \cite{ChaLevLuc11}, and the resulting discrete problem is solved with a primal-dual algorithm with the convolutions implemented through Fourier transforms as in \cite{ChaPoc11}. The boundary conditions were imposed by extending the computational domain and projection onto the corresponding constraint. The results and parameter choices are shown in Figure \ref{fig:tree}.
\begin{figure}[htb]
     \begin{center}
     \mbox{} 
     \hfill
	 \raisebox{-0.5\height}{\includegraphics[width=0.23\textwidth]{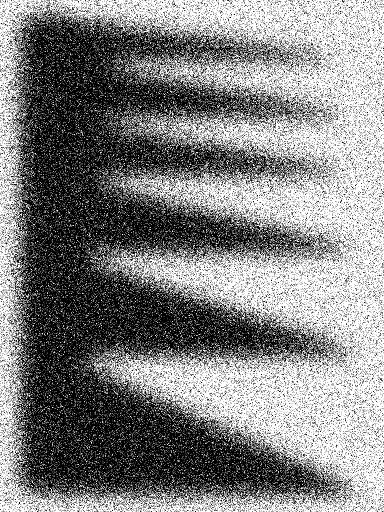}}
     \hfill
	 \raisebox{-0.5\height}{\includegraphics[width=0.23\textwidth]{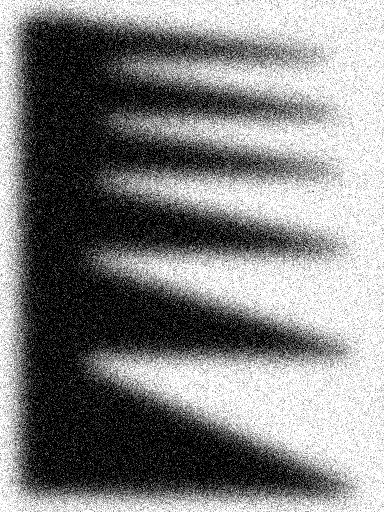}}
     \hfill
	 \raisebox{-0.5\height}{\includegraphics[width=0.23\textwidth]{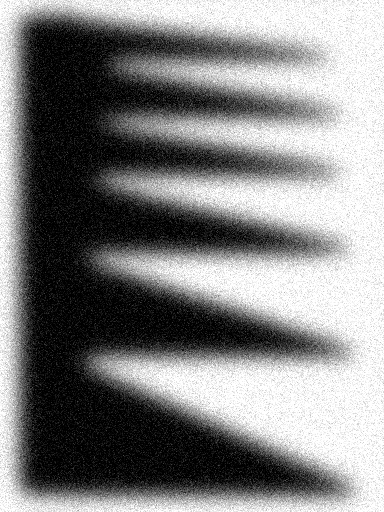}}
	 \hfill
	 \raisebox{-0.5\height}{\includegraphics[width=0.23\textwidth]{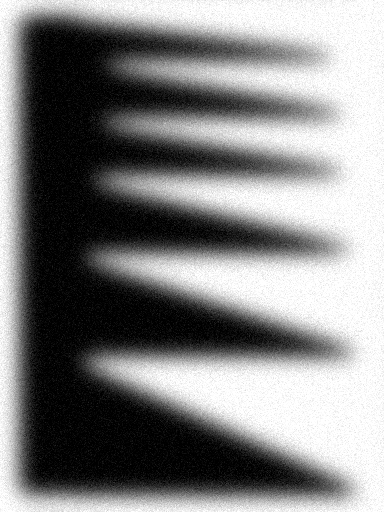}}
	 \hfill
	 \mbox{}
	 \\
	 \vspace{.25cm}
	 \mbox{} 
     \hfill
	 \raisebox{-0.5\height}{\includegraphics[width=0.23\textwidth]{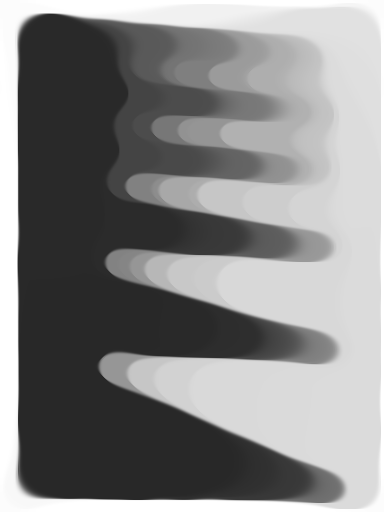}}
     \hfill
	 \raisebox{-0.5\height}{\includegraphics[width=0.23\textwidth]{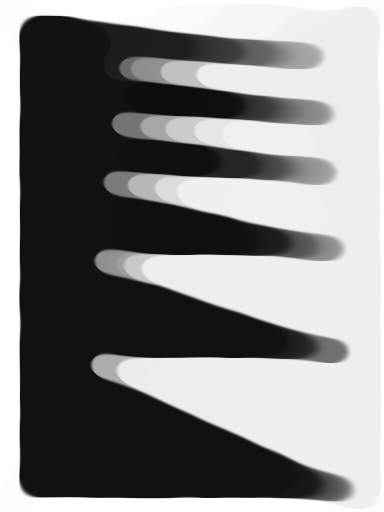}}
     \hfill
	 \raisebox{-0.5\height}{\includegraphics[width=0.23\textwidth]{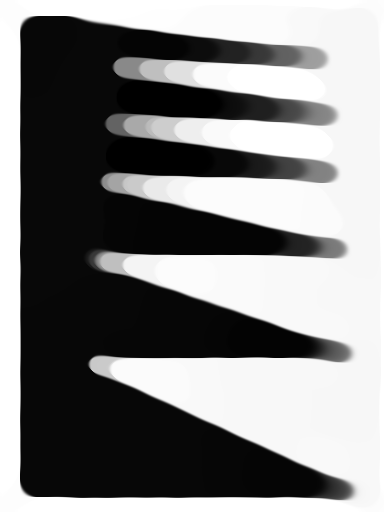}}
	 \hfill
	 \raisebox{-0.5\height}{\includegraphics[width=0.23\textwidth]{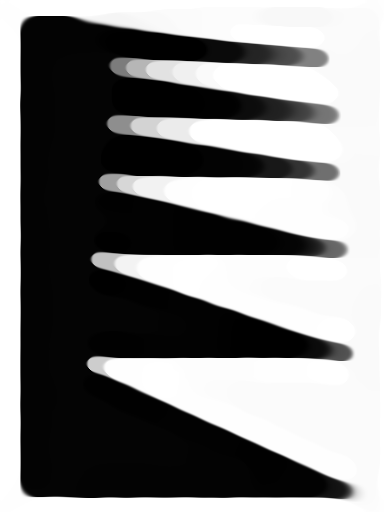}}
	 \hfill
	 \mbox{}
	 \\
	 \vspace{.25cm}
	 \mbox{} 
     \hfill
	 \raisebox{-0.5\height}{\includegraphics[width=0.23\textwidth]{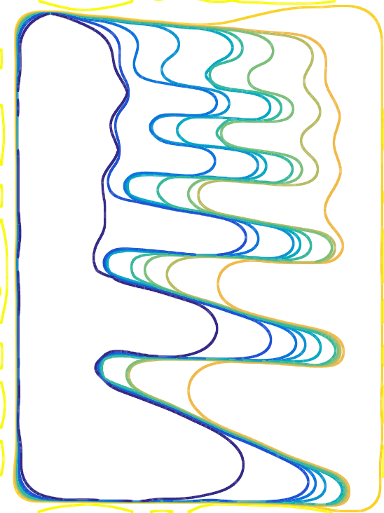}}
     \hfill
	 \raisebox{-0.5\height}{\includegraphics[width=0.23\textwidth]{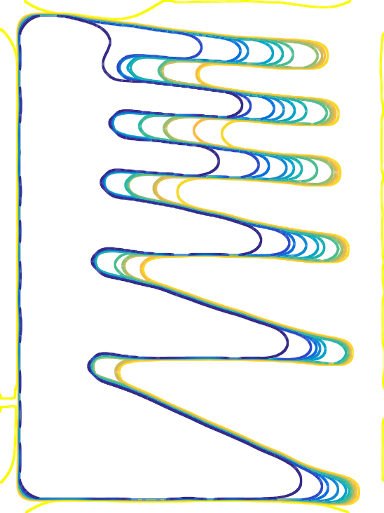}}
     \hfill
	 \raisebox{-0.5\height}{\includegraphics[width=0.23\textwidth]{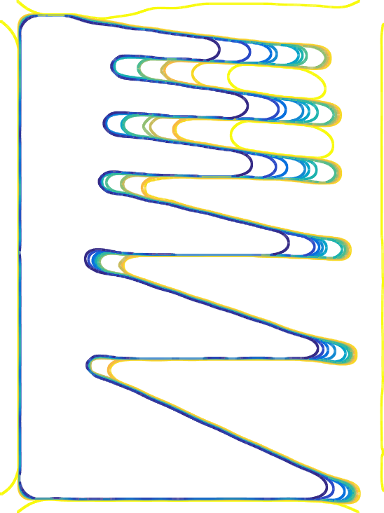}}
	 \hfill
	 \raisebox{-0.5\height}{\includegraphics[width=0.23\textwidth]{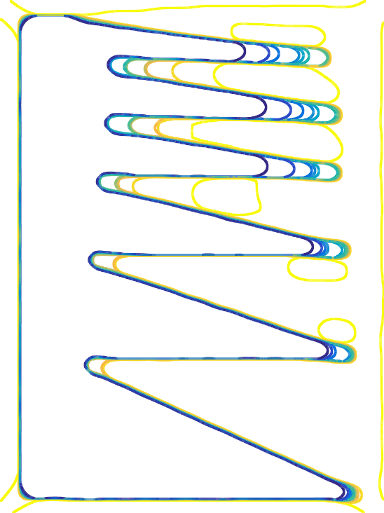}}
	 \hfill
	 \mbox{}
     \end{center}
     \caption{Deblurring of a characteristic function by total variation regularization with Dirichlet boundary conditions. First row: Input image blurred with a known kernel and with additive noise. Second row: numerical deconvolution result, corresponding to minimizers of \eqref{eq:primalDirichlet}. Third row: some level lines of the results. The regularization parameters are $\alpha = 1, 0.25, 0.0625, 0.0156$ and the variance of the Gaussian noise used is $\alpha/10$.}\label{fig:tree}
\end{figure}

\subsection{Denoising in \texorpdfstring{$\R^2$}{R2} or in \texorpdfstring{$\Omega$}{Omega} with Dirichlet conditions}
In this subsection, we consider only denoising ($A = \Id$). If $\udag = f$ has a bounded support in $\R^2$, one can minimize \eqref{eq:primalDirichlet} either in $\R^2$ or in a bounded domain $\Omega$ containing the support of $f$. In general, these minimizations yields different results. Nevertheless, when $\Omega$ is convex, we can easily show
\begin{prop}\label{prop:R2Dir}
Let $f$ have compact support included in an open convex set $\Omega$. Then, minimizing \eqref{eq:primalDirichlet} on $\Omega$ with Dirichlet homogeneous boundary conditions or $\R^2$ lead to the same solution.
\end{prop}
\begin{proof}
 We just need to show that the minimizer $u$ of \eqref{eq:primalDirichlet} in $\R^2$ has a support in $\Omega$. If it were not the case, just note that replacing $u$ by $u \cdot 1_\Omega$ decreases both terms of the functional. For the total variation part, this result uses the convexity of $\Omega$.
\end{proof}
If $\Omega$ is not convex, it is easy to construct examples where this result is no longer true, even for denoising. In Figure \ref{fig:denoiseC}, aggressive total variation denoising is applied to a noiseless image, to illustrate the role of the boundary conditions in the regularization. Nevertheless, the direct application of Theorem 2 show that as $\alpha \to 0$, the level-sets of these two minimizers concentrate around the ones of $f$.

\begin{figure}
\centering
\subfigure{
\hfill \includegraphics[width = 0.28\textwidth]{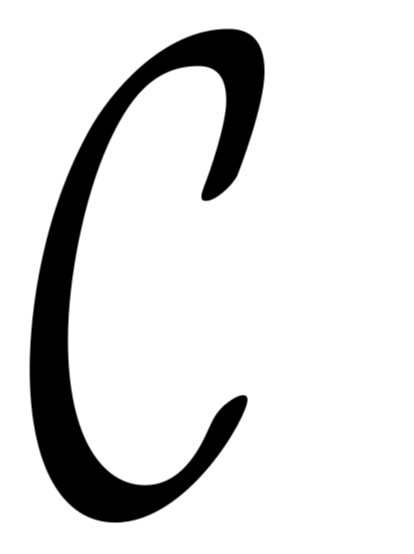} \hspace{0.1cm}
\includegraphics[width = 0.28\textwidth]{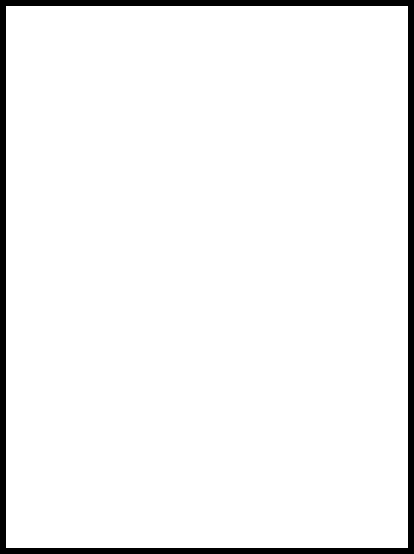} \hspace{0.1cm}
\includegraphics[width = 0.28\textwidth]{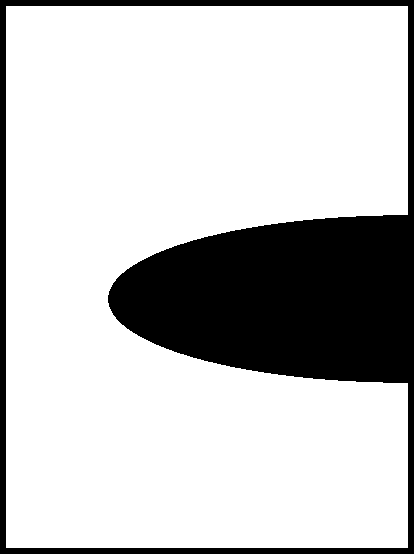} \hfill
}
\vspace{0.1cm}
\subfigure{
\hfill \includegraphics[width = 0.28\textwidth]{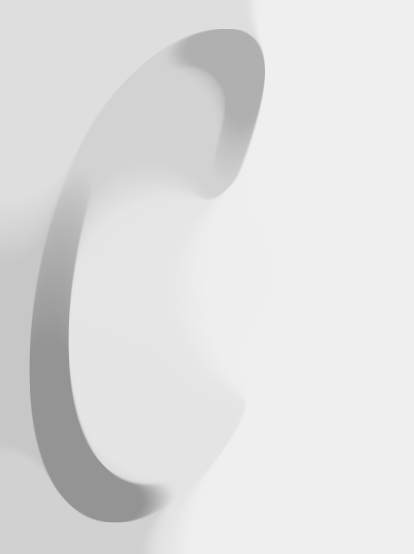} \hspace{0.1cm}
\includegraphics[width = 0.28\textwidth]{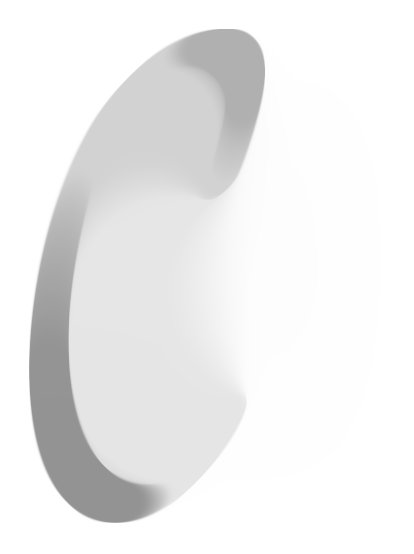} \hspace{0.1cm}
\includegraphics[width = 0.28\textwidth]{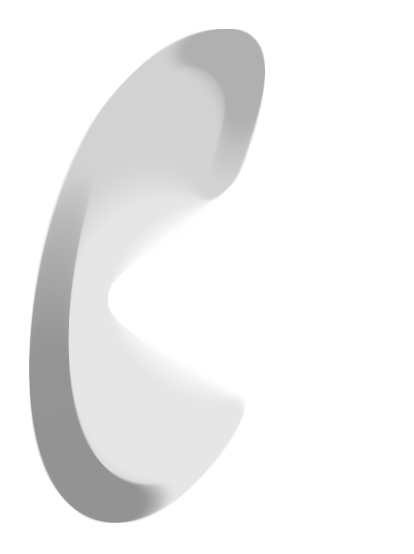} \hfill
}
\caption{Denoising of a $\mathcal C$ with different boundaries and boundary conditions. Upper left: original image. Upper middle: convex Dirichlet domain. Upper right: nonconvex domain. Lower left: result with Neumann boundary. Lower middle: Dirichlet result in the convex domain Lower right: Dirichlet result in the nonconvex domain. The Neumann solution reflects the fact that for level-sets that reach the boundary of $\Omega$, part of their perimeter is not penalized. For the rightmost solution, since the domain is not convex, the solution is different to that of the $\R^2$ case.}\label{fig:denoiseC}
\end{figure}

\subsection{Denoising with Neumann boundary conditions}

As in the Dirichlet case, there are some configurations where solving in a bounded domain does not correspond to solving for $\R^2$. For example, if $A = \Id$, $f = 1_{B(0,1)}$ and $\Omega = B(0,R)$ with $R>1$, the minimizer of \eqref{eq:primalNeumann} is
$$ u_\alpha = \left(1-\alpha-\frac{2\alpha}{R^2-1} \right) 1_{B(0,1)} + \frac{2\alpha}{R^2-1} 1_{B(0,R)},$$
whereas the minimizer of \eqref{eq:primalDirichlet} in $\R^2$ is clearly $1_{B(0,1)}$.

One can also see lower left image in Figure \ref{fig:denoiseC}, which contains the denoising of the $\mathcal C$ in a rectangle.

\section*{Acknowledgments}
The third author was supported by the Austrian Science Fund (FWF) through the National Research Network ``Geometry+Simulation'' (NFN S11704) and the project ``Interdisciplinary Coupled 
Physics Imaging'' (FWF P26687).

\begin{appendix}
\renewcommand{\theequation}{A-\arabic{equation}}
\section{Auxiliary results}
\begin{theorem}[Fenchel duality, \cite{BorZhu05} Theorem 4.4.3]\label{thm:fenchel}
Let $X$ and $Y$ be Banach spaces, $F:X \to \R \cup \{+\infty\}, G:Y \to \R \cup \{+\infty\}$ convex, and $A: X \to Y$ a linear bounded operator. Assume further that there exists a point $x \in X$ such that $F(x) < +\infty$, $G(Ax) < +\infty$ and $G$ is continuous at $Ax$. Then 
$$\inf_{x \in X} \left\{ F(x)+G(Ax)\right\} = \sup_{y^* \in Y^*} \left\{ -F^*(A^*y^*)-G^*(-y^*)\right\},$$
and the supremum is attained.
\end{theorem}

\begin{lemma}\label{lem:subgrad}
Let $X$ be a Hilbert space and $F: X \to \R \cup \{+\infty\}$ a positively 1-homogeneous convex functional. Then, for each $x \in X$ we have
\begin{equation} 
\partial F(x) = \big\{ \xi \in \partial F(0) \ \vert \ \inner{\xi}{x} = F(x) \big\}.\label{eq:subgradx}
\end{equation}
\end{lemma}
\begin{proof}
First one can use convexity and homogeneity to derive a triangle inequality: For any $y,z \in X$ we have 
$$F(y+z)=2 F \left(\frac{y+z}{2}\right)\ls F(y)+F(z).$$
Let us prove \eqref{eq:subgradx}, for which we may assume that $F(x)<+\infty$, since otherwise both sets are empty. If $\xi$ is in the right set, for any $y \in X$, by definition of $\partial F(0)$ and since homogeneity implies $F(0)=0$,
$$F(y)\gs \scal{\xi}{y} \text{ and } -F(x)=-\scal{\xi}{x},$$
which can be summed together to obtain $\xi \in \partial F(x)$.

Now, let $\xi \in \partial F(x)$ and $y \in X$ be arbitrary. The triangle inequality above implies
\begin{equation}\label{eq:subtri}F(y-x) \gs F(y) - F(x) \gs \scal{\xi}{y-x}\end{equation}
since $y$ was arbitrary and $F(0)=0$, this means $\xi \in \partial F(0)$. Using $\xi \in \partial F(0)$, we obtain $F(y) \gs \scal{\xi}{y}$, which subtracted from the second inequality in \eqref{eq:subtri} provides us with $F(x) \ls \scal{\xi}{x}$. The opposite inequality follows again from $\xi \in \partial F(0)$, finishing the proof.
\end{proof}

\begin{prop}[Layer-cake formula,  \cite{LieLos01}, Theorem 1.13]
 Let $u \in L^1(\Omega)$ be nonnegative. Then, one has
 \begin{equation}
  \int_{\Omega} u = \int_0^\infty | \{ u > t \} | \dd t. \label{eq:layercake}
 \end{equation}
\end{prop}

\begin{definition}[Equiintegrability]\label{def:equiint}
Let $(v_s) \subset L^2(\R^2)$ be a family of functions. We say that $(v_s)$ is equiintegrable if for each $\epsilon >0$ there are numbers $\delta >0$ and $R>0$ such that for every measurable subset $F$ with $|F|< \delta$ and all $s$,
\begin{equation*}\left(\int_F |v_s|^2\right)^{1/2} \ls \epsilon \quad \text{and} \quad \left(\int_{\R^2 \setminus B(0,R)} |v_s|^2\right)^{1/2} \ls \epsilon.\end{equation*}
It can be checked directly that a sequence $(v_n)$ that converges strongly in $L^2(\R^2)$ is necessarily equiintegrable.
\end{definition}

\begin{prop}[Coarea formula for functions with bounded variation, \cite{AmbFusPal00}, Theorem 3.40]
 Let $u \in \mathrm{BV}(\Omega)$. Then one has
 \begin{equation} \TVO{u} = \int_{- \infty}^{+ \infty} \per(\{u > t\}) \dd t.\label{eq:coarea}\end{equation}
\end{prop}

\begin{definition}
\label{def:relativeper}
 Since $D1_E$ is a Radon measure, one can consider the perimeter of $E$ in every Borel subset $F$, which we denote by $$P(E\, ;\, F):=|D1_E|(F).$$ 
\end{definition}

\begin{definition}[Reduced boundary and unit normal]
For any set $E$ with finite perimeter in $\Omega$, one can define its reduced boundary $\partial^\ast E$. One says that $x \in \supp(1_E)$ belongs to $\partial^\ast E$ if
$$ \lim_{r \to 0^+} \frac{D1_E(B(x,r))}{|D1_E|(B(x,r))} \quad \text{exists and belongs to } S^1.$$
For $x \in \partial^\ast E$, one can define the measure theoretic outer normal $\nu_E$ to $E$ by
$$ \nu_E(x) := \lim_{r \to 0^+} \frac{D1_E(B(x,r))}{|D1_E|(B(x,r))}.$$
\label{def:redbound}
\end{definition}

\begin{theorem}[De Giorgi, \cite{Giu84}, Theorem 4.4]
For $E$ with finite perimeter in $\Omega$, one has
$$ \per(E \, ; \, \Omega) = \mathcal H^1(\partial^\ast E \cap \Omega).$$
\label{th:degiorgi}
\end{theorem}

\begin{definition}For a Lebesgue set $E \subset \R^2$, we use the notations $E^{(1)}$ and $E^{(0)}$ for the points where the density of $E$ is $1$ and $0$ respectively. That is, for $s\in \{0,1\}$ we have
$$E^{(s)} =\left\{x \in \R^2 \ \middle\vert\ \lim_{r\to 0}\frac{|B(x,r) \cap E|}{|B(x,r)|}=s\right\}.$$
Furthermore, we note that by the Lebesgue differentiation theorem
$$|E^{(0)} \Delta (\R^2 \setminus E)|=0 \text{ and } |E^{(1)} \Delta E|=0.$$
\label{def:densitypoints}
\end{definition}

\begin{theorem}[Federer,  \cite{Mag12}, Theorem 16.2]
Let $E$ have finite perimeter in $\Omega$ and let 
$$ \partial^e E := \R^2 \setminus (E^{(0)} \cap E^{(1)}).$$
Then, $\partial^\ast E \subset \partial^e E$ and
$$ \mathcal H^1(\partial^e E \setminus \partial^\ast E) = 0.$$
\label{th:federer}
\end{theorem}

\begin{theorem}[\cite{Mag12}, Theorem 16.3]
Let $E$ and $F$ be two finite perimeter sets in $\Omega$. Then, for every Borel set $G \subset \R^2$, one has
\begin{equation}
 \per(E \cap F \,;\, G) = \per(E \, ; \, F^{(1)} \cap G) + \per(F \, ; \, E^{(1)} \cap G) + \mathcal H^1 \left( \{\nu_E = \nu_F \} \cap G \right).
 \label{eq:1610}
\end{equation}
and
\begin{equation}
 \per(E \setminus F \,;\, G) = \per(E \, ; \, F^{(0)} \cap G) + \per(F \, ; \, E^{(1)} \cap G) + \mathcal H^1 \left( \{\nu_E = -\nu_F \} \cap G \right).
 \label{eq:1611}
\end{equation}
\end{theorem}
\end{appendix}

\bibliographystyle{plain}
\bibliography{conv_source_cond}
\end{document}